\documentclass{article}
\usepackage[a4paper,top=3cm,bottom=3cm,left=2cm,right=2cm]{geometry}
\usepackage[bookmarks=true, bookmarksopen=true, bookmarksopenlevel=5]{hyperref}
\usepackage{authblk}

\usepackage[english]{babel}
\usepackage{amssymb}
\usepackage{amscd}
\usepackage{amsfonts}
\usepackage{algpseudocode}
\usepackage{algorithm} 
\usepackage{graphicx}
\usepackage{caption}
\usepackage{amsthm} 
\usepackage{amsmath}
\usepackage{microtype}
\usepackage[subrefformat=parens, labelfont=up]{subcaption} 
\usepackage{adjustbox} 
\usepackage{float} 
\usepackage{mathtools}
\usepackage{booktabs}
\usepackage{multirow}
\usepackage{pgfplotstable}
\usepackage{siunitx}
\usepackage{tikz}
\usepackage{pgfplots}
\usepackage[customcolors,norndcorners]{hf-tikz}
\pgfplotsset{compat=newest}
\usepackage{csquotes}
\usepackage[backend=bibtex,style=numeric,indexing,maxbibnames=99,sorting=none,url=false,doi=false]{biblatex} 
\numberwithin{equation}{section}

\DeclareMathOperator{\sspan}{span}
\DeclareMathOperator{\supp}{supp}

\theoremstyle{definition}
\newtheorem{definition}{Definition}
\theoremstyle{plain}

\theoremstyle{plain}

\theoremstyle{plain}

\theoremstyle{plain}
\newtheorem{theorem}{Theorem}
\theoremstyle{definition}
\newtheorem{assumption}{Assumption}
\theoremstyle{remark}
\newtheorem{remark}{Remark}
\theoremstyle{definition}

\newcommand{\ptca}{s}
\newcommand{\ptcb}{r}

\newcommand{\ndof}{N_{\mathrm{dof}}}

\newcommand{\nptc}{N_{\mathrm{patch}}}
\newcommand{\nadj}{N_{\mathrm{adj}}}
\newcommand{\nint}{N_{\mathrm{int}}}

\newcommand{\Xhath}[1]{\,\hspace{-1pt}\widehat{X}^{#1}_h\hspace{-1pt}\,}
\newcommand{\Xhathb}[1]{\,\hspace{-1pt}\widehat{X}^{#1}_{h,0}\hspace{-1pt}\,}

\newcommand{\pad}{\mathbf{P}_{\mathrm{ad}}}
\newcommand{\pdg}{\mathbf{P}_{\mathrm{dg}}}
\newcommand{\logLogSlopeTriangle}[5]
{
	
	\pgfplotsextra
	{
		\pgfkeysgetvalue{/pgfplots/xmin}{\xmin}
		\pgfkeysgetvalue{/pgfplots/xmax}{\xmax}
		\pgfkeysgetvalue{/pgfplots/ymin}{\ymin}
		\pgfkeysgetvalue{/pgfplots/ymax}{\ymax}
		
		\pgfmathsetmacro{\xArel}{#1}
		\pgfmathsetmacro{\yArel}{#3}
		\pgfmathsetmacro{\xBrel}{#1-#2}
		\pgfmathsetmacro{\yBrel}{\yArel}
		\pgfmathsetmacro{\xCrel}{\xArel}
		
		\pgfmathsetmacro{\lnxB}{\xmin*(1-(#1-#2))+\xmax*(#1-#2)} 
		\pgfmathsetmacro{\lnxA}{\xmin*(1-#1)+\xmax*#1} 
		\pgfmathsetmacro{\lnyA}{\ymin*(1-#3)+\ymax*#3} 
		\pgfmathsetmacro{\lnyC}{\lnyA+#4*(\lnxA-\lnxB)}
		\pgfmathsetmacro{\yCrel}{\lnyC-\ymin)/(\ymax-\ymin)} 
		
		\coordinate (A) at (rel axis cs:\xArel,\yArel);
		\coordinate (B) at (rel axis cs:\xBrel,\yBrel);
		\coordinate (C) at (rel axis cs:\xCrel,\yCrel);
		
		\draw[#5]   (A)-- node[pos=0.5,anchor=north] {1}
		(B)-- 
		(C)-- node[pos=0.5,anchor=west] {#4}
		cycle;
	}
}

\begin{document}
\date{}
\title{Isogeometric multi-patch preconditioner for time-dependent electromagnetic problems
}
\author[1]{Bernard Kapidani}
\author[2]{Gabriele Loli}
\author[2,3]{Giancarlo Sangalli}
\author[2,3]{Mattia Tani}
\author[3,4,5]{Rafael V\'azquez}
\affil[1]{\'Ecole Polytechnique F\'ed\'erale Lausanne, Department of Mathematics,  CH-1015 Lausanne, Switzerland}
\affil[2]{Universit\`a di Pavia, Dipartimento di Matematica ``F. Casorati'', Via A. Ferrata 1, 27100 Pavia, Italy}
\affil[3]{Istituto di Matematica Applicata e Tecnologie Informatiche ``E.\ Magenes'' del CNR, Via Ferrata 5, 27100 Pavia, Italy}
\affil[4]{Universidade de Santiago de Compostela, Department of Applied Mathematics, R\'ua Lope G\'omez de Marzoa, s/n, 15782, Santiago de Compostela, Spain}

\affil[5]{Galician Centre for Mathematical Research and Technology, R\'ua de Constantino Candeira, s/n, 15782, Santiago de Compostela, Spain}

\maketitle
\begin{abstract}
We build upon recent developments in high-order spline-based geometric methods for the three-dimensional initial boundary value problem of Maxwell's equations. Previous schemes were limited to single-patch geometries. To address the computational challenges associated with inverting mass matrices in multi-patch settings, we extend existing techniques for 0-forms mass and stiffness matrices to 1-forms mass matrices.
Our approach leverages the tensor-product structure of spline spaces to construct a spectrally equivalent preconditioner for the 1-forms mass matrices. Specifically, we design an efficient preconditioner for single-patch domains that is robust with respect to both mesh size and spline degree. In the multi-patch case, this methodology is extended by integrating the single-patch preconditioner with an additive Schwarz approach, ensuring robustness and scalability with respect to the number of patches.
We analyze the computational complexity of the proposed preconditioner and validate its robustness and efficiency through comprehensive numerical experiments.
\end{abstract}
\section{Introduction} \label{sec:intro}
We are interested in the numerical solution of the initial boundary value problem for Maxwell's equations on a bounded space-time domain $\Omega\times [0, T]$, with $\Omega\subset\mathbb{R}^3$. We follow the notation of exterior calculus (see~\cite{kapidaniHighOrderGeometric2023}) and we write the continuous system of equations in terms of differential forms 
\begin{align} 
& \mathrm{d}\mathrm{H} =  \partial_t \mathrm{D} + \mathrm{J}, \label{eq:ampere}\\
& \mathrm{d}\mathrm{E} = -\partial_t \mathrm{B}, \label{eq:faraday}\\
& \mathrm{d}\mathrm{D} = \rho, \label{eq:divgaugeD} \\
& \mathrm{d}\mathrm{B} = 0, \label{eq:divgaugeB}
\end{align}
where the differential 1-forms $\mathrm{E}$ and $\mathrm{H}$ are respectively the electric and the magnetic field, $\mathrm{D}$ and $\mathrm{B}$ are respectively the electric displacement and the magnetic induction, given as differential 2-forms. Furthermore, $\mathrm{J}$ is a 2-form representing the electric current and $\rho$ is a 3-form representing the electric charge density, both of which act as source terms for the equations. Finally, $\mathrm{d}$ is the exterior derivative, for which a more precise definition will be given in Section~\ref{sec:diffforms}. 

Because the exterior derivative satisfies the nilpotency property $\mathrm{d}^2 = 0$, taking the exterior derivative of \eqref{eq:ampere} and substituting \eqref{eq:divgaugeD} yields the charge continuity equation:
\begin{equation} \label{eq:continuity}
	\partial_t \rho + \mathrm{d}\mathrm{J} = 0.
\end{equation}
This serves as a strict mathematical compatibility condition that the source terms must satisfy.

Equations \eqref{eq:ampere}--\eqref{eq:divgaugeB} have to be completed with constitutive laws, that take the form
\begin{equation} \label{eq:cont_hodge1}
	\begin{aligned}
		& \mathrm{B} = \star^1_\mu \mathrm{H},
		& \mathrm{D} = \star^1_\epsilon \mathrm{E}.
	\end{aligned}
\end{equation}
The $\star^1$ symbol represents the Hodge star operator, which maps 1--forms into 2--forms, that above incorporates the coefficients $\mu$, the magnetic permeability, and $\epsilon$, the electric permittivity, both assumed to be uniformly positive and bounded scalar fields. Again, a more precise definition will be given in Section~\ref{sec:diffforms}.

The framework of exterior calculus admits a natural discretization within isogeometric analysis (IGA), a computational methodology introduced in \cite{Hughes_Cottrell_Bazilevs,IGA-book} as an evolution of the finite element method, in which the unknown fields are represented using splines, NURBS, and their extensions, as the geometric parametrizations used in computer-aided design. A discrete de Rham complex for tensor-product B-splines was first introduced and analyzed in \cite{Buffa_Sangalli_Vazquez,BRSV11}, leading to applications in Galerkin-based discretizations of Maxwell's equations \cite{ratnani,Corno20161}, including those relevant to plasma physics \cite{kraus_kormann_morrison_sonnendrucker_2017,Holderied_2021}. Furthermore, this framework has enabled the development of pointwise divergence-free methods for incompressible fluid simulations \cite{Buffa_deFalco_Sangalli,EvHu12,EvHu12-2,EvHu12-3,Van17}.

Early efforts to incorporate exterior calculus and differential forms into B-spline settings include \cite{back2012}, which employed a dual staggered grid, and \cite{Hiemstra20141444}, where the B-spline-based de Rham complex was adapted for mimetic discretizations in the sense of Bochev and Hyman \cite{bochevPrinciplesMimeticDiscretizations2006}. More recently, the isogeometric de Rham complex has been combined with a discontinuous Galerkin (DG) approach to ensure compatibility across conforming patches in multi-patch geometries \cite{gucluBrokenFEECFramework2022,CamposPintoSchnack2025}.

A novel method leveraging isogeometric differential forms and two dual de Rham complexes was introduced in \cite{kapidaniHighOrderGeometric2022}. Unlike conventional geometric approaches, this method circumvents the explicit construction of a dual mesh by defining the dual complex via a polynomial degree transformation. This idea was initially proposed in \cite{hiemstra_isogeometric_2011} and subsequently applied to achieve stable mortar coupling in non-conforming meshes \cite{Buffa_2020aa,kapidani2021treecotree}. The high continuity of splines ensures a well-defined exterior derivative in both sequences, represented by the incidence matrices of a Cartesian grid \cite{ratnani,BSV14}, while maintaining equal dimensionality between paired spaces. In contrast to the methods in \cite{Hiemstra20141444} and \cite{gucluBrokenFEECFramework2022}, which incorporate a discrete co-derivative (the adjoint operator of the exterior derivative), this approach discretizes the Hodge star operators \cite{HIP99c}, thereby establishing connections between the primal and dual complexes.
The convergence properties of this method were analyzed in \cite{kapidaniHighOrderGeometric2022} for elliptic problems. Moreover, in \cite{kapidaniHighOrderGeometric2023}, the authors demonstrated how to leverage this structure to achieve an efficient, high-order, explicit-in-time approximation of Maxwell's system.

In this paper, we focus on the first discretization scheme introduced in \cite{kapidaniHighOrderGeometric2023}, specifically addressing the efficient solution of linear systems arising from mass matrices in spaces of 1-forms. We propose an efficient solver that preconditions these mass matrices using techniques inspired by \cite{LOLI2022245}, applicable to both single-patch and multi-patch physical domains.

The structure of the paper is as follows. In Section~\ref{sec:diffforms}, we review key concepts of spline complexes for differential forms and introduce two discrete formulations for Maxwell's equations on multi-patch domains: one based on a continuous Galerkin discretization and another on a discontinuous Galerkin-like approach. In Section~\ref{sec:precond}, we propose a preconditioning strategy for the single-patch case and extend it to the multi-patch setting for both formulations. Section~\ref{sec:num} presents numerical results assessing the robustness and efficiency of the proposed preconditioners. Finally, Section~\ref{sec:conclusions} concludes the paper with summarizing remarks.

\section{Discrete spline differential forms and formulations}\label{sec:diffforms}
From this point onward, we focus on the three-dimensional setting, i.e., $n = 3$, and consider a domain $\Omega \subset \mathbb{R}^3$. Given a non-negative integer $k$, let $\Lambda^k(\Omega)$ denote the space of smooth differential $k$-forms. To extend this notion, we introduce the Hilbert space $L^2\Lambda^k(\Omega)$, defined as the completion of $\Lambda^k(\Omega)$ with respect to the $L^2$-inner product (see \cite{AFW06}).

The exterior derivative, denoted by $\mathrm{d}^k$, maps differential $k$-forms to $(k+1)$-forms. In the Hilbert space framework, this operator is interpreted as a densely defined, closed operator from $L^2\Lambda^k(\Omega)$ to $L^2\Lambda^{k+1}(\Omega)$ (see \cite{AFW-2}). For simplicity, when the order of the form is clear from the context, we write $\mathrm{d}$ instead of $\mathrm{d}^k$, as seen in equations \eqref{eq:ampere}--\eqref{eq:divgaugeB}. A key property of this operator is that for any differential $k$-form $\omega$, the composition satisfies $\mathrm{d} \circ \mathrm{d} \omega = 0$.

Following \cite{AFW06}, we define the Sobolev spaces of differential forms for $k = 0, \dots, 3$ as:
\begin{equation*}
	H\Lambda^k(\Omega) = \left\{ \omega \in L^2\Lambda^k(\Omega) \mid \mathrm{d}^k \omega \in L^2\Lambda^{k+1}(\Omega) \right\}.
\end{equation*}
These spaces enable the formulation of the $L^2$ de Rham complex of differential forms:
\begin{equation} \label{eq:derhamcomplex}
	\begin{CD}
		H\Lambda^0(\Omega) @>\mathrm{d}^0>> H\Lambda^1(\Omega) @>\mathrm{d}^1>> H\Lambda^2(\Omega) @>\mathrm{d}^2>> H\Lambda^3(\Omega).
	\end{CD}
\end{equation}
Moreover, we consider the de Rham complex for differential forms with vanishing boundary traces. This involves compactly supported differential $k$-forms, whose spaces are denoted with a subscript zero. They form the sequence:
\begin{equation} \label{eq:derhamcomplex_homogeneous}
	\begin{CD}
		H_0\Lambda^0(\Omega) @>\mathrm{d}^0>> H_0\Lambda^1(\Omega) @>\mathrm{d}^1>> H_0\Lambda^2(\Omega) @>\mathrm{d}^2>> H_0\Lambda^3(\Omega).
	\end{CD}
\end{equation}
If the domain $\Omega$ is enclosed by a perfect electrical conductor, the tangential component of the electric field must vanish on $\partial \Omega$. Consequently, equations \eqref{eq:ampere}--\eqref{eq:divgaugeB} must be solved for 1-forms $\mathrm{E} \in H_0 \Lambda^1(\Omega)$ and $\mathrm{H} \in H \Lambda^1(\Omega)$, as well as for 2-forms $\mathrm{B} \in H_0 \Lambda^2(\Omega)$ and $\mathrm{D} \in H \Lambda^2(\Omega)$. For numerical approximation, discrete subspaces are employed: the fields $\mathrm{E}$ and $\mathrm{B}$ belong to subspaces of \eqref{eq:derhamcomplex_homogeneous}, while $\mathrm{H}$ and $\mathrm{D}$ belong to discrete subspaces of \eqref{eq:derhamcomplex}. The system is further subject to initial conditions:
\begin{equation*}
	\mathrm{D}(\boldsymbol{x},0) = \mathrm{D}_0 (\boldsymbol{x}), \quad \mathrm{B}(\boldsymbol{x},0) = \mathrm{B}_0 (\boldsymbol{x}).
\end{equation*}

We will work with spline discretizations as standard in IGA, and we
will make the assumption that the domain is connected and described by
the union of non-singular maps of cubes, i.e., we define the physical
domain $\overline{\Omega} = \bigcup \overline{\Omega}_{\ptca} \subset \mathbb{R}^3$ through
a family of parametrizations of the form $\textbf{F}_{\ptca}: \widehat
\Omega \rightarrow \Omega_{\ptca}$, $s=1,\ldots,\nptc$, where $\widehat \Omega = (0,1)^3$
is the parametric domain. We thus relate differential
$k$-forms in the parametric domain to differential $k$-forms in the
physical domain using a set of pullback operators $\iota_{\ptca}^k: H
\Lambda^k(\Omega_{\ptca}) \rightarrow H\Lambda^k(\widehat \Omega)$. In
particular we are interested in:%
\begin{equation}\label{eq:pullback}
  \iota_{\ptca}^1 ( \cdot) = (D\textbf{F}_{\ptca})^{\top} (\cdot
  \, \circ \textbf{F}_{\ptca} ),
\end{equation}
where $D\textbf{F}_{\ptca}$ denotes the Jacobian matrix of the mapping $\textbf{F}_{\ptca}$.
An important property is that the pullback commutes both with the exterior derivative and the wedge product.
These tools allow us to define the primal and dual complex of splines, as introduced in \cite{Buffa_2020aa}. In the following section we present the primal and dual complex of splines for the three-dimensional case. We refer to \cite{Buffa_2020aa} for more details, and to \cite{kapidaniHighOrderGeometric2022} for a presentation in terms of differential forms.
\subsection{Isogeometric discrete differential forms}
\subsubsection{Single patch}
In what follows, we establish notation by reviewing the fundamentals of the tensor-product B-spline complex of isogeometric discrete differential forms, see~\cite{kapidaniHighOrderGeometric2022} for details.

Consider a continuous B-spline basis~$\mathcal{B}_p(\Xi):=\{ \widehat{B}_{i,p} \}_{i=1}^{m}$ of degree~$p$ defined on a $p$-open knot vector $\Xi$ according to the Cox-De Boor recursion formula~\cite{DeBoor}. We denote by $S_p(\Xi)$ the space spanned by this basis, which consists of piecewise polynomials of degree $p$ on $(0,1)$. Given a vector of polynomial degrees ${\mathbf{p}} = (p_1, p_2, p_3)$ and a set of open knot vectors ${\boldsymbol \Xi} = \left\{ \Xi_\ell \right\}_{\ell=1}^3$, we define the trivariate B-spline basis as
\begin{equation*}
	{\mathcal B}_{\mathbf{p}}(\boldsymbol \Xi) := \{ \widehat{B}_{\mathbf{i},p}({\widehat{\boldsymbol{x}}}) = \widehat{B}_{i_1,p_1}(\widehat{x}_1) \widehat{B}_{i_2,p_2}(\widehat{x}_2) \widehat{B}_{i_3,p_3}(\widehat{x}_3) \} \; \textup{ for ${\widehat{\boldsymbol{x}}} \in (0,1)^3$},
\end{equation*}
with the corresponding space they span as
\begin{equation*}
	S_{\mathbf{p}}(\boldsymbol \Xi) := \otimes^3_{\ell=1} S_{p_\ell}\left(\Xi_\ell\right) = \sspan ({\mathcal B}_{\mathbf{p}}(\boldsymbol \Xi)).
\end{equation*}
Following the notation in \cite{evans_hierarchical_2020}, for a spline degree $p_\ell$ and a knot vector $\Xi_\ell$, we define
\begin{equation*}
	\left( \widetilde{p}_\ell, \widetilde{\Xi}_\ell \right) \equiv 	\left( \widetilde{p}_\ell(k), \widetilde{\Xi}_\ell(k) \right) := 
	\begin{cases}
		\left( p_\ell - 1, \Xi'_\ell \right) & \textup{ if } \ell = k, \\
		\left( p_\ell , \Xi_\ell \right) & \textup{ otherwise},
	\end{cases}
\end{equation*}
where $\Xi'_{\ell}$ is obtained from $\Xi_{\ell}$ by removing the first and last repeated knots.\\
For $k=1,2,3$ we introduce 
\begin{equation*}
	S_{\mathbf{p}}(\boldsymbol \Xi;k) := \otimes^3_{\ell=1} S_{\widetilde{p}_\ell}(\widetilde{\Xi}_\ell; k)=\otimes^3_{\ell=1}\sspan (\mathcal{B}_{\widetilde{p}_\ell}(\widetilde{\Xi}_\ell; k)),
\end{equation*}
where each basis $\mathcal{B}_{\widetilde{p}_\ell}(\widetilde{\Xi}_\ell; k)$ is defined as
\begin{equation*}
	\mathcal{B}_{\widetilde{p}_\ell}(\widetilde{\Xi}_\ell; k) := 
	\begin{cases}
		\mathcal{B}_{p_\ell-1}(\Xi'_\ell) & \textup{ if } \ell = k, \\
		\mathcal{B}_{p_\ell}(\Xi_\ell) & \textup{ otherwise}.
	\end{cases}
\end{equation*}
As a basis of $S_{\mathbf{p}}(\boldsymbol \Xi;k)$ we choose
\begin{equation*}
	{\mathcal B}_{\mathbf p}(\boldsymbol \Xi; k) := \{\widehat{B}_{{\mathbf i}, {\mathbf p}}(\widehat{\boldsymbol x}) = \widehat{B}_{i_1,\widetilde{p}_1}(\widehat{x}_1) \widehat{B}_{i_2,\widetilde{p}_2}(\widehat{x}_2) \widehat{B}_{i_3,\widetilde{p}_3}(\widehat{x}_3),\, \widehat{\boldsymbol x} \in (0,1)^3: \widehat{B}_{i_\ell,\widetilde{p}_\ell} \in \mathcal{B}_{\widetilde{p}_\ell}(\widetilde{\Xi}_\ell; k) \}.
\end{equation*}
The basis of B-splines differential $1$-forms is then defined as
\begin{equation}\label{eq:basis}
	{\mathcal B}^1 := \bigcup_{k=1}^{3} \left\{ \widehat{B}_{\mathbf i,\mathbf p}(\widehat{\boldsymbol x}) \mathrm{d}\widehat{x}_{k} : \widehat{B}_{\mathbf i,\mathbf p} \in {\mathcal B}_{\mathbf p}(\boldsymbol \Xi;k) \right\}.
\end{equation}
To the multi-index $\mathbf{i}=(i_1,i_2,i_3)$ in~\eqref{eq:basis} we associate a scalar index $i$, corresponding to the lexicographical ordering of the degrees-of-freedom, given by
\begin{equation*}
	i=i_1+(i_2-1)m_{k,1} + (i_3-1)m_{k,1}m_{k,2},
\end{equation*}
where $m_{k,\ell}=\dim(S_{p_\ell}(\Xi_\ell; k))$. Thus, we rewrite
\begin{equation*}
	\mathcal{B}^1 = \bigcup_{k=1}^{3} \left\{ \widehat{B}_{i,\mathbf{p}}(\widehat{\boldsymbol x}) \mathrm{d}\widehat{x}_{k} : \widehat{B}_{i,\mathbf{p}} \in {\mathcal B}_{\mathbf p}(\boldsymbol \Xi;k) \right\}.
\end{equation*}
We denote the tensor-product spline space of isogeometric discrete differential $1$-forms by
\begin{equation*}
	\Xhath{1}(\widehat{\Omega}) := \left\{ \widehat{\omega}^h = \sum_{k=1}^{3} \widehat{\omega}^h_k \mathrm{d}\widehat{x}_{k} :  \widehat{\omega}^h_{k} \in S_{\mathbf p}(\boldsymbol \Xi;k), \ k=1,2,3 \right\} = \sspan\left({\mathcal B}^1\right).
\end{equation*}
Using the isomorphism between differential forms and vector proxy fields, we have
\begin{align*}
	\Xhath{1}(\widehat{\Omega}) &\simeq S_{\mathbf p}\left(\boldsymbol \Xi; 1\right) \times S_{\mathbf p}\left(\boldsymbol \Xi; 2\right) \times S_{\mathbf p}\left(\boldsymbol \Xi; 3\right) \\
	&= S_{p_1-1,p_2,p_3}(\Xi'_1,\Xi_2,\Xi_3) \times S_{p_1,p_2-1,p_3}(\Xi_1,\Xi'_2,\Xi_3) \times S_{p_1,p_2,p_3-1}(\Xi_1,\Xi_2,\Xi'_3).
\end{align*}
To handle spaces with vanishing boundary conditions, we define $\Xhathb{1} := \Xhath{1} \cap H_0 \Lambda^1(\Omega)$. In particular, $\Xhathb{1}$ is constructed analogously to $\Xhath{1}$, except that we remove from $\mathcal{B}_{p_\ell}(\Xi_\ell)$ the two basis functions not vanishing at the boundary of $(0,1)$, while $\mathcal{B}_{p_\ell-1}(\Xi'_\ell)$ remains unchanged, i.e.
\begin{equation}\label{eq:space_bc}
	\Xhathb{1}=\sspan\left({\mathcal B}_0^1\right),
\end{equation}
where
\begin{equation*}
	\mathcal{B}_0^1 := \bigcup_{k=1}^{3} \left\{ \widehat{B}_{i,\mathbf{p}}(\widehat{\boldsymbol x}) \mathrm{d}\widehat{x}_{k} : \widehat{B}_{i,\mathbf{p}} \in {\mathcal B}_{\mathbf{p},0}(\boldsymbol \Xi;k) \right\},
\end{equation*}
\begin{equation*}
	{\mathcal B}_{\mathbf{p},0}(\boldsymbol \Xi; k) := \{\widehat{B}_{{\mathbf i}, {\mathbf p}}(\widehat{\boldsymbol x}) = \widehat{B}_{i_1,\widetilde{p}_1}(\widehat{x}_1) \widehat{B}_{i_2,\widetilde{p}_2}(\widehat{x}_2) \widehat{B}_{i_3,\widetilde{p}_3}(\widehat{x}_3),\, \widehat{\boldsymbol x} \in (0,1)^3: \widehat{B}_{i_\ell,\widetilde{p}_\ell} \in \mathcal{B}_{\widetilde{p}_\ell,0}(\widetilde{\Xi}_\ell; k) \},
\end{equation*}
\begin{equation*}
	\mathcal{B}_{\widetilde{p}_\ell,0}(\widetilde{\Xi}_\ell; k) := 
	\begin{cases}
		\mathcal{B}_{p_\ell-1}(\Xi'_\ell) & \textup{ if } \ell = k, \\
		\mathcal{B}_{p_\ell}(\Xi_\ell)\cap H^1_0(0,1) & \textup{ otherwise}.
	\end{cases}
\end{equation*}
We assume that the physical domain $\Omega \subset \mathbb{R}^3$ is parametrized by a map $\textbf{F}: \widehat{\Omega} \rightarrow \Omega$ which is a bi-Lipschitz homeomorphism, and arbitrarily smooth when restricted to every element, as in \cite{BBSV-acta}. Then, the discrete space of isogeometric differential 1-forms in $\Omega$ is defined using the pullback operator $\iota^1(\cdot) = (D\textbf{F})^{\top} (\cdot \circ \textbf{F})$, where $D\textbf{F}$ denotes the Jacobian matrix of $\textbf{F}$. Moreover, without loss of generality, we can assume that $\textbf{F}$ is orientation-preserving, meaning the determinant of its Jacobian is strictly positive, $\det(D\textbf{F}) > 0$. Specifically, we define:
\begin{equation*}
	X_h^1(\Omega) := \left\{ \omega : \iota^1(\omega) \in \Xhath{1}(\widehat{\Omega}) \right\}, \qquad
	X_{h,0}^1(\Omega) := \left\{ \omega : \iota^1(\omega) \in \Xhathb{1}(\widehat{\Omega}) \right\}.
\end{equation*}
For more details, see~\cite{BSV14}.

Finally, following~\cite{Buffa_2020aa}, we define the dual spline space of discrete differential forms $\widetilde{X}_h^1(\Omega)$. The construction is completely analogous to that of the primal spline space of discrete differential forms ${X}_h^1(\Omega)$, with the only difference that the spline degrees are reduced by one in each parametric direction. In particular, we have
\begin{equation*}
	\widetilde{X}_h^1(\Omega):= \left\{\, \omega : \iota^1(\omega) \in \widehat{\widetilde{X^1_h}}(\widehat{\Omega}) \right\},
\end{equation*}
with
\begin{equation*}
	\widehat{\widetilde{X^1_h}}(\widehat{\Omega}) \simeq  	S_{p_1-2,p_2-1,p_3-1}\!\left(\Xi''_1,\Xi'_2,\Xi'_3\right) \times S_{p_1-1,p_2-2,p_3-1}\!\left(\Xi'_1,\Xi''_2,\Xi'_3\right) \times S_{p_1-1,p_2-1,p_3-2}\!\left(\Xi'_1,\Xi'_2,\Xi''_3\right),
\end{equation*}
where the knot vector $\Xi''_i$ is obtained from $\Xi'_i$ by removing one repetition of both the first and the last knot.

\subsubsection{Multi-patch}
In the multi-patch setting, $\Omega$ is given as the union of several disjoint patches:
\begin{equation*}
	\overline{\Omega} = \bigcup_{\ptca=1}^{\nptc} \overline{\Omega}_{\ptca},
\end{equation*}
where each patch $\Omega_{\ptca}$ is parametrized from the reference domain $\widehat{\Omega}$ via a mapping $\mathbf{F}_{\ptca}$, for $\ptca = 1, \dots, \nptc$.  
As in the single-patch case, we associate to each patch $\Omega_{\ptca}$ a space of isogeometric differential $0$-forms, $X_h^0(\Omega_{\ptca})$, and a space of differential $1$-forms, $X_h^1(\Omega_{\ptca})$.
The $H \Lambda^1$-conforming spaces of isogeometric differential $1$-forms over the entire domain $\Omega$ are then defined as
\begin{equation}\label{eq:multi-patch_spaces}
	X_{h,0}^1 := \Biggl( \bigtimes_{\ptca=1}^{\nptc} X_{h}^1(\Omega_\ptca) \Biggr) \cap H_0\Lambda^1(\Omega), 
	\qquad
	\widetilde{X}_{h}^1 := \Biggl( \bigtimes_{\ptca=1}^{\nptc} \widetilde{X}_{h}^1(\Omega_\ptca) \Biggr) \cap H\Lambda^1(\Omega),
\end{equation}
where $\bigtimes$ denotes the Cartesian product.  
\begin{assumption}\label{ass:conformity}
On each non-empty interface $\Gamma_{\ptca,\ptcb} := \partial\Omega_{\ptca} \cap \partial\Omega_{\ptcb}$, with $\ptca \neq \ptcb$, we assume that the trace spaces coincide:
\begin{equation}
	\mathrm{tr}(X_h^1(\Omega_{\ptca})|_{\Gamma_{\ptca,\ptcb}} = 
	\mathrm{tr}(X_h^1(\Omega_{\ptcb})|_{\Gamma_{\ptca,\ptcb}}.
\end{equation}
\end{assumption}
Under Assumption~\ref{ass:conformity}, tangential continuity across patches can be enforced by merging coincident degrees of freedom along interfaces, as detailed in~\cite{BSV14}.  
For the space $X_h^1(\Omega)$, this merging procedure must also account for the orientation of the degrees of freedom. 

To construct spaces with vanishing boundary conditions on multi-patch domains, we proceed patch-wise following the procedure described for the single-patch case. When doing so, basis functions supported on interior patch interfaces must not be removed.

Taking both the merging procedure and the boundary constraints into account, we define the global multi-patch basis $\mathcal{B}_{\mathrm{mp}}$ of $X_{h,0}^1(\Omega)$ as
\begin{equation}
	\mathcal{B}_{\mathrm{mp}} := \left\{ \beta_i \right\}_{i=1}^{\ndof},
\end{equation}
such that, for each patch $\Omega_{\ptca}$, either $\iota^1_{\ptca}\left( \beta_i|_{\Omega_{\ptca}} \right)$ or $\iota^1_{\ptca}\left( -\beta_i|_{\Omega_{\ptca}} \right)$ belongs to the local basis ${\mathcal{B}}^1_{\ptca}$.

\subsection{Continuous Galerkin multi-patch formulation} \label{sec:cg}
We now present a continuous Galerkin formulation for the discretization of the problem in~\eqref{eq:ampere}--\eqref{eq:divgaugeB}. Since this approach relies on conforming spline spaces, we consider the spaces introduced in~\eqref{eq:multi-patch_spaces}.

By employing the two specific discrete Hodge--star operators described in~\cite{kapidaniHighOrderGeometric2023}, the continuous Galerkin semi-discrete-in-space variational formulation reads: find $\mathrm{E}_h \in X_{h,0}^1$ and $\mathrm{H}_h \in \widetilde{X}_{h}^1$ such that
\begin{equation}\label{eq:cg_formulation}
	\begin{cases}
		\begin{aligned}
			\left( \partial_t \mathrm{E}_h , \eta_h \right)_{L_\epsilon^2 \Lambda^1(\Omega)}
			&= \phantom{-}\int_\Omega \left( \mathrm{d}\mathrm{H}_h - \mathrm{J}_h \right) \wedge \eta_h,
			&& \forall\, \eta_h \in X_{h,0}^1,\\[4pt]
			\left( \partial_t \mathrm{H}_h , \xi_h \right)_{L_\mu^2 \Lambda^1(\Omega)}
			&= -\int_\Omega \mathrm{d}\mathrm{E}_h \wedge \xi_h,
			&& \forall\, \xi_h \in \widetilde{X}_{h}^1,
		\end{aligned}
	\end{cases}
\end{equation}
where $(\cdot, \cdot)_{L^2_{\gamma}} \Lambda^1(\Omega)$ is the inner product in $L^2 \Lambda^1(\Omega)$ with $\gamma$ a uniformly positive weight, see \cite{AFW06}.

We observe that, when solving~\eqref{eq:cg_formulation} with an explicit time-marching scheme, one must invert linear systems associated with the mass matrix of $X_{h,0}^1$, denoted ${\mathbf{M}}_\epsilon$, and with the mass matrix of $\widetilde{X}_{h}^1$, denoted $\widetilde{\mathbf{M}}_\mu$.
Finally, we note that energy conservation for this semi-discrete scheme follows directly, as shown in~\cite{kapidaniHighOrderGeometric2023}. For this and additional details, we refer the reader to that reference.
\subsection{Symmetric discontinuous Galerkin formulation} \label{sec:sdg}
We now introduce a symmetric discontinuous Galerkin formulation for the problem in~\eqref{eq:ampere}--\eqref{eq:divgaugeB}. In contrast to the continuous Galerkin setting, no inter-patch continuity is enforced here. Accordingly, we consider the local spline spaces on each patch and define the global discrete spaces as the Cartesian products:
\begin{equation*}
	\mathrm{E}_h \in \bigtimes_{\ptca=1}^{\nptc} X_h^1(\Omega_{\ptca}),
	\qquad
	\mathrm{H}_h \in \bigtimes_{\ptca=1}^{\nptc} \widetilde{X}_h^1(\Omega_{\ptca}).
\end{equation*}
Since continuity across patches is not imposed, the corresponding mass matrices are block diagonal.  
With a harmless abuse of notation, we denote them again by ${\mathbf{M}}_\epsilon$ and $\widetilde{\mathbf{M}}_\mu$, although they differ from the mass matrices introduced for the continuous Galerkin formulation.

The semi-discrete variational formulation then reads: find $\mathrm{E}_h \in \bigtimes_\ptca X_h^1(\Omega_\ptca)$ and $\mathrm{H}_h \in \bigtimes_\ptca \widetilde{X}_h^1(\Omega_\ptca)$ such that 
\begin{equation}\label{eq:sdg_formulation}
\begin{cases}
\begin{aligned}
	\sum\limits_{\ptca=1}^{\nptc}
	\left( \partial_t \mathrm{E}_h , \eta_h \right)_{L_\epsilon^2 \Lambda^1(\Omega_\ptca)}&=
	\sum\limits_{\ptca=1}^{\nptc} \left(\int_{\Omega_\ptca} \mathrm{H}_h \wedge \mathrm{d}\eta_h + \oint_{\partial\Omega_\ptca} \mathrm{H}_h^* \wedge \eta_h -	\int_{\Omega_\ptca} \mathrm{J}_h \wedge \eta_h \right), && \forall \eta_h \in X_h^1(\Omega_\ptca),\\
	\sum\limits_{\ptca=1}^{\nptc} \left( \partial_t \mathrm{H}_h , \xi_h \right)_{L_\mu^2 \Lambda^1(\Omega_\ptca)} &= \sum\limits_{\ptca=1}^{\nptc} \left(\int_{\Omega_\ptca} \mathrm{d}\mathrm{E}_h \wedge \xi_h +  \oint_{\partial\Omega_\ptca} \left( \mathrm{E}_h^* - \mathrm{E}_h \right) \wedge \xi_h \right), && \forall \xi_h \in \widetilde{X}_h^1(\Omega_\ptca).
\end{aligned}
\end{cases}
\end{equation}
Here, $\mathrm{E}_h^*$ and $\mathrm{H}_h^*$ denote numerical fluxes defined on the inter-patch boundaries.  
When central numerical fluxes (i.e., arithmetic averages across patch interfaces, see~\cite{hesthavenNodalDiscontinuousGalerkin2008}) are employed, the resulting semi-discrete scheme preserves the electromagnetic energy.
\section{Preconditioning}\label{sec:precond}
As stated in Section~\ref{sec:diffforms}, for both the continuous and discontinuous Galerkin formulations, we are led to solve linear systems associated with mass matrices for $X^1_{h,0}$ and $\widetilde{X}^1_h$, namely ${\mathbf M}_\epsilon$ and $\widetilde{\mathbf{M}}_\mu$. We present in this section how to exploit the tensor-product structure of the spline spaces for devising a preconditioner for those matrices.
We base our work on \cite{LOLI2022245,sangalliIsogeometricPreconditionersBased2016}, where results were shown for the 0-forms mass and stiffness matrices. We here adapt the tools introduced therein to the case of 1-forms.
For brevity, we will only discuss how to handle ${\mathbf M}_\epsilon$, as for $\widetilde{\mathbf M}_\mu$, a similar approach can be followed. 
\subsection{Single patch preconditioner}\label{sec:single_patch_prec}
Let us start by considering a single patch domain. In this case the matrix ${\mathbf{M}}_\epsilon$ has the block structure
\begin{equation}
\label{eq:mass_eps}	{\mathbf M}_\epsilon =
	\begin{bmatrix}
		\mathbf{M}_{1,1} & \mathbf{M}_{1,2} & \mathbf{M}_{1,3} \\
		\mathbf{M}^{\top}_{1,2} & \mathbf{M}_{2,2} & \mathbf{M}_{2,3} \\
		\mathbf{M}^{\top}_{1,3} & \mathbf{M}^{\top}_{2,3} & \mathbf{M}_{3,3}
	\end{bmatrix},
\end{equation}
with the blocks $\mathbf{M}_{k,\ell}$ defined by
\begin{align}
	\nonumber {[\mathbf{M}_{k,\ell}]}_{i,j} &:= \left( (\iota^1)^{-1}(\widehat{B}_{i,\mathbf{p}}\mathrm{d}\widehat{{x}}_k), (\iota^1)^{-1}(\widehat{B}_{j,\mathbf{p}}\mathrm{d}\widehat{{x}}_\ell) \right)_{L_\epsilon^2 \Lambda^1(\Omega)} \\
\label{eq:mass_eps_kl}	&= \int_{\Omega} \widehat{B}_{i,\mathbf{p}}(\textbf{F}^{-1}(\boldsymbol{x})) \widehat{B}_{j,\mathbf{p}} (\textbf{F}^{-1}(\boldsymbol{x})) \epsilon(\boldsymbol{x})[D\textbf{F}^{-1}(\textbf{F}^{-1}(\boldsymbol{x}))D\textbf{F}^{-\top}(\textbf{F}^{-1}(\boldsymbol{x}))]_{k,\ell}\, \mathrm{d}\boldsymbol{x}\\
\nonumber	&= \int_{\widehat \Omega} \widehat{B}_{i,\mathbf{p}}(\widehat{\boldsymbol{x}}) \widehat{B}_{j,\mathbf{p}} (\widehat{\boldsymbol{x}}) \mathcal{C}_{k,\ell}(\widehat{\boldsymbol{x}})\, \mathrm{d}\widehat{\boldsymbol{x}},
\end{align}
for $\widehat{B}_{i,\mathbf{p}} \in {\mathcal B}_{\mathbf p}(\boldsymbol \Xi;k)$ and $\widehat{B}_{j,\mathbf{p}} \in {\mathcal B}_{\mathbf p}(\boldsymbol \Xi;\ell)$, and using the notation $ \widehat{\boldsymbol{x}} = \textbf{F}^{-1}(\boldsymbol{x})$ and
\begin{equation*}
	\mathcal{C}_{k,\ell}(\widehat{\boldsymbol{x}})=\widehat{\epsilon}(\widehat{\boldsymbol{x}})[D\textbf{F}^{-1}(\widehat{\boldsymbol{x}})D\textbf{F}^{-\top}(\widehat{\boldsymbol{x}})]_{k,\ell}\det(D\textbf{F}(\widehat{\boldsymbol{x}})),
\end{equation*}
with $\widehat{\epsilon} = \epsilon \circ \textbf{F}$.
\begin{definition}
Let $\mathbb{S}^3_+$ denote the set of symmetric positive definite $3 \times 3$ matrices, equipped with the Euclidean norm. Given a function $\mathcal{A} \in L^{\infty}(\widehat{\Omega},\mathbb{S}^3_+)$ such that
\begin{equation*}
	\inf_{\widehat{\boldsymbol{x}} \in \widehat{\Omega}}\left[\lambda_{\mathrm{min}}\left(\mathcal{A}(\widehat{\boldsymbol{x}})\right)\right] >0,
\end{equation*}
we define the preconditioner $\mathbf{P}_{\mathcal{A}}$ as the symmetric and positive definite matrix associated to the bilinear form
\begin{equation*}
	\mathbf{u}^{\top} \mathbf{P}_{\mathcal{A}} \mathbf{v} = \sum_{k,\ell=1}^{3} \int_{\widehat{\Omega}} u_{k}(\widehat{\boldsymbol{x}}) v_{\ell}(\widehat{\boldsymbol{x}}) \mathcal{A}_{k,\ell}(\widehat{\boldsymbol{x}}) \, \mathrm{d}\widehat{\boldsymbol{x}}
\end{equation*}
where $\mathbf{u} = \begin{bmatrix} \mathbf{u}_1^{\top} & \mathbf{u}_2^{\top} & \mathbf{u}_3^{\top} \end{bmatrix}^{\top}$, $\mathbf{v} = \begin{bmatrix} \mathbf{v}_1^{\top} & \mathbf{v}_2^{\top} & \mathbf{v}_3^{\top} \end{bmatrix}^{\top}$, and for $k=1,2,3,$ $\mathbf{u}_k,\mathbf{v}_k$ are the coordinate vectors of $u_k,v_k \in \sspan(\mathcal{B}_{\mathbf{p}}(\boldsymbol \Xi;k)) $ with respect to the basis $\mathcal{B}_{\mathbf{p}}(\boldsymbol \Xi;k)$.

\end{definition}

The next theorem shows that ${\mathbf M}_\epsilon$ and $\mathbf{P}_{\mathcal{A}}$ are spectrally equivalent.

\begin{theorem}\label{thm:spectral_equiv}
There exist two positive constants $C_1$ and $C_2$, independent of $\mathbf{\Xi}$ and $\mathbf{p}$, such that  
\begin{equation}\label{eq:spectral_equiv}
	C_1 \leq \lambda_{\mathrm{min}}\left(\mathbf{P}_{\mathcal{A}}^{-\frac{1}{2}} {\mathbf M}_\epsilon \mathbf{P}_{\mathcal{A}}^{-\frac{1}{2}}\right) 
	\leq \lambda_{\mathrm{max}}\left(\mathbf{P}_{\mathcal{A}}^{-\frac{1}{2}} {\mathbf M}_\epsilon \mathbf{P}_{\mathcal{A}}^{-\frac{1}{2}}\right) 
	\leq C_2.
\end{equation}
\end{theorem}
\begin{proof}
A generic 1-form in $X^1_{h,0}$ can be seen as $(\iota^1)^{-1}(v)$, with $v \in \Xhathb{1}(\Omega)$ defined as
\begin{equation*}
	v =\sum_{k=1}^3 v_{k} \mathrm{d}\widehat{x}_k, \qquad v_{k} \in \sspan(\mathcal{B}_{\mathbf{p}}(\boldsymbol \Xi;k)).
\end{equation*}
For $k=1,2,3$, let $\mathbf{v}_k$ be the coordinate vector of $v_{k}$ with respect to the basis ${\mathcal B}_{\mathbf p}(\boldsymbol \Xi;k)$.
By the Courant-Fischer theorem, \eqref{eq:spectral_equiv} is equivalent to find $C_1$ and $C_2$ such that
\begin{equation*}
	C_1 \leq \frac{\mathbf{v}^{\top} {\mathbf M}_\epsilon\mathbf{v}}{\mathbf{v}^{\top} \mathbf{P}_{\mathcal{A}} \mathbf{v}} \leq C_2, \qquad \text{for any} \qquad \mathbf{v} = [\mathbf{v}_1^{\top},\mathbf{v}_2^{\top},\mathbf{v}_3^{\top}]^{\top}.
\end{equation*}
Using \eqref{eq:mass_eps} and \eqref{eq:mass_eps_kl}, we observe that 
\begin{equation*}
	\mathbf{v}^{\top} {\mathbf M}_\epsilon\mathbf{v} = \sum_{k,\ell=1}^{3} \int_{\widehat{\Omega}} v_{k}(\widehat{\boldsymbol{x}}) v_{\ell}(\widehat{\boldsymbol{x}}) \mathcal{C}_{k,\ell}(\widehat{\boldsymbol{x}}) \, \mathrm{d}\widehat{\boldsymbol{x}}.
\end{equation*}
Due to the regularity and orientation-preserving assumptions on $\textbf{F}$, and the positivity of $\epsilon$, for all $\widehat{\boldsymbol{x}} \in \widehat \Omega$, the $3 \times 3$ matrix 
\begin{equation*}
	\mathcal{C}(\widehat{\boldsymbol{x}}) := \begin{bmatrix}
		\mathcal{C}_{1,1}(\widehat{\boldsymbol{x}}) & \mathcal{C}_{1,2}(\widehat{\boldsymbol{x}}) & \mathcal{C}_{1,3}(\widehat{\boldsymbol{x}}) \\
		\mathcal{C}_{2,1}(\widehat{\boldsymbol{x}}) & \mathcal{C}_{2,2}(\widehat{\boldsymbol{x}}) & \mathcal{C}_{2,3}(\widehat{\boldsymbol{x}}) \\
		\mathcal{C}_{3,1}(\widehat{\boldsymbol{x}}) & \mathcal{C}_{3,2}(\widehat{\boldsymbol{x}}) & \mathcal{C}_{3,3}(\widehat{\boldsymbol{x}})
	\end{bmatrix}
\end{equation*}
is symmetric positive definite and we can write
\begin{align*}
	\inf_{\widehat{\boldsymbol{x}} \in \widehat{\Omega}}\left[\lambda_{\mathrm{min}}\left(\mathcal{C}(\widehat{\boldsymbol{x}}) \mathcal{A}^{-1}(\widehat{\boldsymbol{x}})\right)\right]	\mathbf{v}^{\top} \mathbf{P}_{\mathcal{A}}\mathbf{v} &\leq \sum_{k,\ell=1}^{3} \int_{\widehat{\Omega}} v_{k}(\widehat{\boldsymbol{x}}) v_{\ell}(\widehat{\boldsymbol{x}}) \mathcal{A}_{k,\ell}(\widehat{\boldsymbol{x}})  \frac{\displaystyle \sum_{k',\ell'=1}^{3} v_{k'}(\widehat{\boldsymbol{x}}) v_{\ell'}(\widehat{\boldsymbol{x}}) {\mathcal{C}}_{k',\ell'}(\widehat{\boldsymbol{x}})}{ \displaystyle \sum_{k',\ell'=1}^{3} v_{k'}(\widehat{\boldsymbol{x}}) v_{\ell'}(\widehat{\boldsymbol{x}}) \mathcal{A}_{k',\ell'}(\widehat{\boldsymbol{x}})} \, \mathrm{d}\widehat{\boldsymbol{x}} \\
	&=\mathbf{v}^{\top} {\mathbf{M}}_{\epsilon}\mathbf{v}.
\end{align*}	
Similarly, 
\begin{align*}
	\mathbf{v}^{\top} {\mathbf{M}}_{\epsilon}\mathbf{v} &= \sum_{k,\ell=1}^{3} \int_{\widehat{\Omega}} v_{k}(\widehat{\boldsymbol{x}}) v_{\ell}(\widehat{\boldsymbol{x}}) \mathcal{A}_{k,\ell}(\widehat{\boldsymbol{x}}) \frac{ \displaystyle \sum_{k',\ell'=1}^{3} v_{k'}(\widehat{\boldsymbol{x}}) v_{\ell'}(\widehat{\boldsymbol{x}}) {\mathcal{C}}_{k',\ell'}(\widehat{\boldsymbol{x}})}{ \displaystyle \sum_{k',\ell'=1}^{3} v_{k'}(\widehat{\boldsymbol{x}}) v_{\ell'}(\widehat{\boldsymbol{x}}) \mathcal{A}_{k',\ell'}(\widehat{\boldsymbol{x}})} \, \mathrm{d}\widehat{\boldsymbol{x}} \\
	&\leq \sup_{\widehat{\boldsymbol{x}} \in \widehat{\Omega}}\left[\lambda_{\mathrm{max}}\left(\mathcal{C}(\widehat{\boldsymbol{x}}) \mathcal{A}^{-1}(\widehat{\boldsymbol{x}})\right)\right]	\mathbf{v}^{\top} \mathbf{P}_{\mathcal{A}}\mathbf{v}.
\end{align*}	
Hence, the thesis holds with
\begin{equation*}
	C_1 = \inf_{\widehat{\boldsymbol{x}} \in \widehat{\Omega}}\left[\lambda_{\mathrm{min}}\left(\mathcal{C}(\widehat{\boldsymbol{x}}) \mathcal{A}^{-1}(\widehat{\boldsymbol{x}})\right)\right] \qquad
	C_2 = \sup_{\widehat{\boldsymbol{x}} \in \widehat{\Omega}}\left[\lambda_{\mathrm{max}}\left(\mathcal{C}(\widehat{\boldsymbol{x}}) \mathcal{A}^{-1}(\widehat{\boldsymbol{x}})\right)\right].
\end{equation*}
\end{proof}
\begin{remark}\label{rem:conditioning}
Theorem~\ref{thm:spectral_equiv} provides the following bound on the condition number of ${\mathbf M}_\epsilon$ under symmetric preconditioning by $\mathbf{P}_{\mathcal{A}}$:
\begin{equation*}
	\kappa_2 \left( \mathbf{P}_{\mathcal{A}}^{-\frac{1}{2}} {\mathbf M}_\epsilon \mathbf{P}_{\mathcal{A}}^{-\frac{1}{2}} \right) 
	\leq 
	\frac{\displaystyle \sup_{\widehat{\boldsymbol{x}} \in \widehat{\Omega}} \lambda_{\mathrm{max}} \left( \mathcal{C}(\widehat{\boldsymbol{x}}) \mathcal{A}^{-1}(\widehat{\boldsymbol{x}}) \right)}
	{\displaystyle \inf_{\widehat{\boldsymbol{x}} \in \widehat{\Omega}} \lambda_{\mathrm{min}} \left( \mathcal{C}(\widehat{\boldsymbol{x}}) \mathcal{A}^{-1}(\widehat{\boldsymbol{x}}) \right)},
\end{equation*}
with the right-hand side independent of both $\mathbf{\Xi}$ and $\mathbf{p}$.  
As a consequence, the number of iterations required by the preconditioned conjugate gradient (PCG) method to solve a system with ${\mathbf M}_\epsilon$ remains uniformly bounded with respect to $\mathbf{\Xi}$ and $\mathbf{p}$. Moreover, the bound becomes tighter as $\mathcal{A}(\widehat{\boldsymbol{x}})$ more closely approximates $\mathcal{C}(\widehat{\boldsymbol{x}})$.
\end{remark}

Following~\cite{LOLI2022245}, we propose applying a symmetric diagonal scaling to further enhance the performance of $\mathbf{P}_{\mathcal{A}}$, yielding:
\begin{equation*}
	\overline{\mathbf{P}}_{\mathcal{A}} := \mathbf{D}_{\mathcal{A}} \mathbf{P}_{\mathcal{A}} \mathbf{D}_{\mathcal{A}},
\end{equation*}
where $\mathbf{D}_{\mathcal{A}}$ is a diagonal matrix with entries defined as
\begin{equation*}
	[\mathbf{D}_{\mathcal{A}}]_{i,i} := \left( \frac{[{\mathbf M}_\epsilon]_{i,i}}{[\mathbf{P}_{\mathcal{A}}]_{i,i}} \right)^{\frac{1}{2}}.
\end{equation*}

In the following, we discuss two possible choices for $\mathcal{A}$ to construct robust and efficient preconditioners.
\subsubsection{Block diagonal preconditioner}\label{sec:bd_prec}
With the choice $\mathcal{A} = \mathbf{I}$, the preconditioner $\mathbf{P}_{\mathbf{I}}$ reduces to a block-diagonal matrix
\begin{equation}\label{eq:prec_sp}
	\mathbf{P}_{\mathbf{I}} :=
	\begin{bmatrix}
		\widehat{\mathbf{M}}_{1,1} & \mathbf{0} & \mathbf{0} \\
		\mathbf{0} & \widehat{\mathbf{M}}_{2,2} & \mathbf{0} \\
		\mathbf{0} & \mathbf{0} & \widehat{\mathbf{M}}_{3,3}
	\end{bmatrix},
\end{equation}
where we have set
$\widehat{\mathbf{M}}_{k,k} := \widehat{\mathbf{M}}_{3;k,k} \otimes \widehat{\mathbf{M}}_{2;k,k} \otimes \widehat{\mathbf{M}}_{1;k,k}$, for $k=1,2,3$, with
\begin{equation*}
	[\widehat{\mathbf{M}}_{d;k,k}]_{i,j} = \int_{0}^{1} \widehat{B}_{i,\widetilde{p}_d} \widehat{B}_{j,\widetilde{p}_d}\, \mathrm{d}\widehat{x}, \qquad \widehat{B}_{i,p_d}, \widehat{B}_{j,p_d} \in \mathcal{B}_{p_d}(\Xi_d; k), \ d=1,2,3.
\end{equation*}
We emphasize that linear systems associated with $\widehat{\mathbf{M}}_{k,k}$, $k=1,2,3$, can be solved efficiently by exploiting its Kronecker structure, see e.g.~\cite{LOLI2022245}. To sum up, to solve a system associated with $\overline{\mathbf{P}}_{\mathbf{I}}$ we can follow Algorithm~\ref{alg:P_I}.
\begin{algorithm}[ht]
	\caption{Solution of the block diagonal preconditioner system $\overline{\mathbf{P}}_{\mathbf{I}} \mathbf{y} = \mathbf{z}$}\label{alg:P_I}
	\begin{algorithmic}[1]
		\State Assemble the matrices $\widehat{\mathbf{M}}_{d;k,k}$, for $d,k=1,2,3,$ and compute the diagonal scaling $\mathbf{D}_{\mathbf{I}}$. 
		\State Compute $\widetilde{\mathbf{z}}= \begin{bmatrix} \widetilde{\mathbf{z}}_1^\top & \widetilde{\mathbf{z}}_2^\top & \widetilde{\mathbf{z}}_3^\top \end{bmatrix}^\top = \mathbf{D}_{\mathbf{I}}^{-1} \mathbf{z}$.
		\State Solve $\left(\widehat{\mathbf{M}}_{3;k,k} \otimes \widehat{\mathbf{M}}_{2;k,k} \otimes \widehat{\mathbf{M}}_{1;k,k}\right) \widetilde{\mathbf{y}}_k = \widetilde{\mathbf{z}}_k$, for $k=1,2,3$.
		\State Compute $\mathbf{y} = \mathbf{D}_{\mathbf{I}}^{-1} \widetilde{\mathbf{y}}$, with $\widetilde{\mathbf{y}} = \begin{bmatrix} \widetilde{\mathbf{y}}_1^\top & \widetilde{\mathbf{y}}_2^\top & \widetilde{\mathbf{y}}_3^\top \end{bmatrix}^\top$.
	\end{algorithmic}
\end{algorithm}
\begin{remark}\label{rem:trivial_F}
	In the special case where $\mathbf{F}$ is the identity map and $\epsilon = 1$, then $\mathbf{P}_{\mathbf{I}}$ coincides with ${\mathbf M}_\epsilon$.
\end{remark}
\subsubsection{Block Kronecker preconditioner}\label{sec:fd_prec}
In many interesting real-world problems, the physical domain $\Omega$ can be parametrized with a map $\textbf{F}$ verifying $\partial_i\textbf{F} \perp \partial_3\textbf{F}$ for $i=1,2$. This occurs in general in 2.5-dimensional geometries, for example, with solids obtained by extrusion, revolution, or by sweeping a two-dimensional section along a path perpendicular to the cross section.  In this case it holds $\mathcal{C}_{1,3}=\mathcal{C}_{2,3}=\mathcal{C}_{3,1}=\mathcal{C}_{3,2}=0$, and ${\mathbf M}_\epsilon$ reduces to
\begin{align}\label{eq:block_structure}
{\mathbf M}_\epsilon &=
\begin{bmatrix}
	\mathbf{M}^{}_{1,1} & \mathbf{M}^{}_{1,2} & \mathbf 0 \\
	\mathbf{M}^{\top}_{1,2} & \mathbf{M}_{2,2} & \mathbf 0 \\
	\mathbf 0 & \mathbf 0 & \mathbf{M}^{}_{3,3}
\end{bmatrix} \\
&=
\begin{bmatrix}
	\mathbf I & \mathbf 0 & \mathbf 0 \\
	\mathbf{M}^{\top}_{1,2} \mathbf{M}^{-1}_{1,1} & \mathbf I & \mathbf 0 \\
	\mathbf 0 & \mathbf 0 & \mathbf I
\end{bmatrix}
\begin{bmatrix}
	\mathbf{M}_{1,1} & \mathbf 0 & \mathbf 0 \\
	\mathbf 0 & \mathbf{S} & \mathbf 0 \\
	\mathbf 0 & \mathbf 0 & \mathbf{M}_{3,3}
\end{bmatrix}
\begin{bmatrix}
	\mathbf I & \mathbf{M}^{-1}_{1,1} \mathbf{M}_{1,2} & \mathbf 0 \\
	\mathbf 0 & \mathbf I & \mathbf 0 \\
	\mathbf 0 & \mathbf 0 & \mathbf I
\end{bmatrix},
\end{align}
where we have set
\begin{equation*}
\mathbf{S} = \mathbf{M}_{2,2} - {\mathbf{M}^{\top}_{1,2}} \mathbf{M}^{-1}_{1,1} \mathbf{M}_{1,2}.
\end{equation*}
To effectively address these cases, we propose a preconditioner $\mathbf{P}_{\overline{\mathcal{C}}}$, where $\overline{\mathcal{C}}$ has the same structure as $\mathcal{C}$, namely,
\begin{equation*}
	\overline{\mathcal{C}} = \begin{bmatrix}
		\overline{\mathcal{C}}_{1,1} & \overline{\mathcal{C}}_{1,2} & 0 \\
		\overline{\mathcal{C}}^{\top}_{1,2} & \overline{\mathcal{C}}_{2,2} & 0 \\
		0 & 0 & \overline{\mathcal{C}}_{3,3}
	\end{bmatrix},
\end{equation*}
where each entry $\overline{\mathcal{C}}_{k,\ell}(\widehat{\boldsymbol{x}})$ is a separable approximation of $\mathcal{C}_{k,\ell}(\widehat{\boldsymbol{x}})$. More precisely,
\begin{equation}\label{eq:separable_factorization}
	\mathcal{C}_{k,\ell}(\widehat{\boldsymbol{x}}) \approx \overline{\mathcal{C}}_{k,\ell}(\widehat{\boldsymbol{x}}) := \omega^{(1)}_{k,\ell}(\widehat{x}_1) \omega^{(2)}_{k,\ell}(\widehat{x}_2) \omega^{(3)}_{k,\ell}(\widehat{x}_3). 
\end{equation}
for $(k,\ell) \in \left\{ (1,1), (1,2), (2,1), (2,2), (3,3) \right\}$.
\begin{assumption}\label{ass:pos_def}
We assume that the approximate coefficient matrix $\overline{\mathcal{C}}(\widehat{\boldsymbol{x}})$ is positive definite for any $\widehat{\boldsymbol{x}} \in \widehat{\Omega}$, which corresponds to the conditions
\begin{equation} \label{pos_def_ass1}
	\overline{\mathcal{C}}_{k,k}(\widehat{\boldsymbol{x}}) = \omega^{(1)}_{k,k}(\widehat{x}_1) \omega^{(2)}_{k,k}(\widehat{x}_2) \omega^{(3)}_{k,k}(\widehat{x}_3) > 0, \qquad k =1,2,3, 	
\end{equation}
and
\begin{equation} \label{pos_def_ass2}
	\overline{\mathcal{C}}_{1,1}(\widehat{\boldsymbol{x}})	\overline{\mathcal{C}}_{2,2}(\widehat{\boldsymbol{x}}) -  \left( \overline{\mathcal{C}}_{1,2}(\widehat{\boldsymbol{x}}) \right)^2  = \prod_{d=1}^{3} \omega^{(d)}_{1,1}(\widehat{x}_d) \omega^{(d)}_{2,2}(\widehat{x}_d) - \prod_{d=1}^{3} \left( {\omega^{(d)}_{1,2}} (\widehat{x}_d)\right)^2 > 0,
\end{equation}
for any $\widehat{\boldsymbol{x}} \in  \widehat{\Omega}$
\end{assumption}
With this choice, the preconditioner $\mathbf{P}_{\overline{\mathcal{C}}}$ exhibits the following block Kronecker structure:
\begin{equation*}
\begin{aligned}
\mathbf{P}_{\overline{\mathcal{C}}} =&
\begin{bmatrix}
	\widehat{\mathbf{M}}_{1,1} & \widehat{\mathbf{M}}_{1,2} & \mathbf 0 \\
	\widehat{\mathbf{M}}^{\top}_{1,2} & \widehat{\mathbf{M}}_{2,2} & \mathbf 0 \\
	\mathbf 0 & \mathbf 0 & \widehat{\mathbf{M}}_{3,3}
\end{bmatrix} \\
=&
\begin{bmatrix}
	\mathbf I & \mathbf 0 & \mathbf 0 \\
	\widehat{\mathbf{M}}^{\top}_{1,2} \widehat{\mathbf{M}}^{-1}_{1,1} & \mathbf I & \mathbf 0 \\
	\mathbf 0 & \mathbf 0 & \mathbf I
\end{bmatrix}
\begin{bmatrix}
	\widehat{\mathbf{M}}_{1,1} & \mathbf 0 & \mathbf 0 \\
	\mathbf 0 & \widehat{\mathbf{S}} & \mathbf 0 \\
	\mathbf 0 & \mathbf 0 & \widehat{\mathbf{M}}_{3,3}
\end{bmatrix}
\begin{bmatrix}
	\mathbf I & \widehat{\mathbf{M}}^{-1}_{1,1} \widehat{\mathbf{M}}_{1,2} & \mathbf 0 \\
	\mathbf 0 & \mathbf I & \mathbf 0 \\
	\mathbf 0 & \mathbf 0 & \mathbf I
\end{bmatrix},
\end{aligned}
\end{equation*}
where, for $(k,\ell) \in \left\{ (1,1), (1,2), (2,1), (2,2), (3,3) \right\}$, we have defined:
\begin{equation*}
	\widehat{\mathbf{M}}_{k,\ell} := \widehat{\mathbf{M}}_{3;k,\ell} \otimes \widehat{\mathbf{M}}_{2;k,\ell} \otimes \widehat{\mathbf{M}}_{1;k,\ell},
\end{equation*}
with
\begin{equation*}
	[\widehat{\mathbf{M}}_{d;k,\ell}]_{i,j} := \int_{0}^{1} \widehat{B}_{i,\widetilde{p}_d} \widehat{B}_{j,\widetilde{p}_d} \omega^{(d)}_{k,\ell}\, \mathrm{d}\widehat{x},
\end{equation*}
for $d=1,2,3$, and for $\widehat{B}_{i,\widetilde{p}_d} \in \mathcal{B}_{\widetilde{p}_d}(\widetilde{\Xi}_d; k)$, and ${\widehat{B}_{j,\widetilde{p}_d} \in \mathcal{B}_{\widetilde{p}_d}(\widetilde{\Xi}_d; \ell)}$; and moreover
\begin{equation*}
	\widehat{\mathbf{S}} := \widehat{\mathbf{M}}_{3;2,2} \otimes \widehat{\mathbf{M}}_{2;2,2} \otimes \widehat{\mathbf{M}}_{1;2,2} - \widehat{\mathbf{T}}_3 \otimes \widehat{\mathbf{T}}_2 \otimes \widehat{\mathbf{T}}_1, 
\end{equation*}
with $\widehat{\mathbf{T}}_d := \widehat{\mathbf{M}}^{\top}_{d;1,2} \left( \widehat{\mathbf{M}}_{d;1,1} \right)^{-1} \widehat{\mathbf{M}}_{d;1,2} $ for $d=1,2,3$.
Note that Assumption~\ref{ass:pos_def} guarantees that $\mathbf{P}_{\overline{\mathcal{C}}}$, and thus $\widehat{\mathbf{S}}$, are positive definite.  We also emphasize that to denote the diagonal blocks of $\mathbf{P}_{\overline{\mathcal{C}}}$ and their Kronecker factors we are using the same notation already used in the case of $\mathbf{P}_{\mathbf{I}}$, even though the definition of these matrices is different. This will not create confusion in the following, since this notation will not be used anymore after the current section.
\begin{remark}
Similarly as for $\mathbf{P}_{\mathbf{I}}$, if $\mathbf{F}$ is the identity map and $\epsilon = 1$, it follows that also $\mathbf{P}_{\overline{\mathcal{C}}}$ coincides with ${\mathbf M}_\epsilon$.
\end{remark}
\begin{remark}
It is of course possible to use the preconditioner ${\mathbf{P}}_{\overline{\mathcal{C}}}$ even when $\partial_3 \mathbf{F}$ is not orthogonal to $\partial_1 \mathbf{F}$ and $\partial_2 \mathbf{F}$, i.e., when $\mathcal{C}$ has no identically zero entries. In accordance with Theorem~\ref{rem:conditioning}, $\mathbf{P}_{\overline{\mathcal{C}}}$ still provides an effective preconditioning strategy as long as the entries $\overline{\mathcal{C}}_{k,\ell}$, for $1 \leq k < \ell \leq 3$, are relatively small. One such case is considered in the numerical experiments of Section~\ref{sec:num}.
\end{remark}
Since the quantity $\omega^{(d)}_{k,\ell}$ is used solely for computing the integrals that define $\widehat{\mathbf{M}}_{d;k,\ell}$, our main focus is on obtaining the approximation~\eqref{eq:separable_factorization} at the quadrature nodes. We denote by $\mathbf{C}_{k,\ell} \in \mathbb{R}^{q^{(1)}_{k,\ell} \times q^{(2)}_{k,\ell} \times q^{(3)}_{k,\ell}}$ the three-dimensional tensor whose entries correspond to the evaluations of the function $\overline{\mathcal{C}}_{k,\ell}$ at the quadrature nodes. Here, $q^{(d)}_{k,\ell}$, for $d = 1,2,3$, represents the number of quadrature nodes in the $d$-th parametric direction.
We first consider the case $k = \ell$, for which all entries of $\mathbf{C}_{k,k}$ are strictly positive. In this scenario, the evaluations $\boldsymbol{\omega}^{(d)}_{k,k}$ at the quadrature nodes, corresponding to the separable form~\eqref{eq:separable_factorization}, can be computed via a few iterations of Algorithm~\ref{alg:sep_approx}, see~\cite{diliberto1951}. This procedure, in particular, guarantees that condition~\eqref{pos_def_ass1} is satisfied.
For the case $k \neq \ell$, we instead employ a rank-1 \emph{canonical polyadic decomposition} (CPD), see~\cite{Hitchcock1927TheEO}. It is important to note that, after computing this approximation, condition~\eqref{pos_def_ass2} must be explicitly verified, as it is not automatically ensured. Nevertheless, in all our numerical experiments, this condition was consistently satisfied.
\begin{algorithm}[ht]
	\caption{Separable approximation of $\mathbf{C}_{k,k} \in \mathbb{R}^{q^{(1)}_{k,k} \times q^{(2)}_{k,k} \times q^{(3)}_{k,k}}$~\cite{diliberto1951}}\label{alg:sep_approx}
	\begin{algorithmic}[1]
		\State Initialize $\boldsymbol{\omega}^{(d)}_{k,k}$ as the vector of length $q^{(d)}_{k,k}$ with all values equal to 1, for $d=1,2,3$. 
		\For{iter=1,\ldots,maxIter}
		\For{$d=1,2,3$}
		\State Compute $\mathcal{V}_d \in \mathbb{R}^{q^{(1)}_{k,k} \times q^{(2)}_{k,k} \times q^{(3)}_{k,k}}$ s.t.
		\begin{equation*}
			[\mathcal{V}_d]_{i_1,i_2,i_3}=\frac{[\mathbf{C}_{k,k}]_{i_1,i_2,i_3}}{\displaystyle \prod_{\substack{s=1,2,3 \\s \neq d}}\left[\boldsymbol{\omega}^{(s)}_{k,k}\right]_{i_s}}.
		\end{equation*}
		\For{$j=1,\ldots,q^{(d)}_{k,k}$}
		\State Compute $m=\min\left\{ [\mathcal{V}_d]_{i_1,i_{2},i_3}\, \big| \, i_d=j, \, i_s=1,\ldots,q^{(s)}_{k,k},\, \text{for } s\left\{ 1,2,3\right\} \ \text{and}\ s \neq d \right\}$.
		\State Compute $M=\max\left\{ [\mathcal{V}_d]_{i_1,i_{2},i_3}\, \big| \, i_d=j, \, i_s=1,\ldots,q^{(s)}_{k,k},\, \text{for } s \in \left\{ 1,2,3\right\} \ \text{and} \ s \neq d \right\}$.
		\State Update $\left[\boldsymbol{\omega}^{(d)}_{k,k}\right]_{j}= \sqrt{mM}$.
		\EndFor
		\EndFor
		\EndFor		
	\end{algorithmic}
\end{algorithm}	
\\
We observe that the inverse of $\mathbf{P}_{\overline{\mathcal{C}}}$ can be factorized as
\begin{equation*}
	\mathbf{P}_{\overline{\mathcal{C}}}^{-1} = \begin{bmatrix}
		\mathbf I & -\widehat{\mathbf{M}}^{-1}_{1,1} \widehat{\mathbf{M}}_{1,2} & \mathbf 0 \\
		\mathbf 0 & \mathbf I & \mathbf 0 \\
		\mathbf 0 & \mathbf 0 & \mathbf I
	\end{bmatrix}
	\begin{bmatrix}
		\widehat{\mathbf{M}}_{1,1}^{-1} & \mathbf 0 & \mathbf 0 \\
		\mathbf 0 & \widehat{\mathbf{S}}^{-1} & \mathbf 0 \\
		\mathbf 0 & \mathbf 0 & \widehat{\mathbf{M}}_{3,3}^{-1}
	\end{bmatrix}
	\begin{bmatrix}
		\mathbf I & \mathbf 0 & \mathbf 0 \\
		-\widehat{\mathbf{M}}^{\top}_{1,2} \widehat{\mathbf{M}}^{-1}_{1,1} & \mathbf I & \mathbf 0 \\
		\mathbf 0 & \mathbf 0 & \mathbf I
	\end{bmatrix} := \mathbf{U} \mathbf{B}^{-1} \mathbf{L}.
\end{equation*}
The application of $\overline{\mathbf{P}}^{-1}_{\overline{\mathcal{C}}}$, which entails solving a linear system associated with $\overline{\mathbf{P}}_{\overline{\mathcal{C}}}$, requires the inversion of the matrices $\widehat{\mathbf{S}}$ and $\widehat{\mathbf{M}}_{k,k}$ for $k=1,3$, as well as the computation of matrix-vector products involving $\widehat{\mathbf{M}}_{1,2}$ and $\widehat{\mathbf{M}}_{1,2}^{\top}$. The linear systems associated with $\widehat{\mathbf{M}}_{k,k}$, for $k=1,3$, can be solved using the same strategy described at the end of Section~\ref{sec:bd_prec}.
To invert $\widehat{\mathbf{S}}$, we employ the \emph{fast diagonalization} (FD) algorithm described in~\cite{sangalliIsogeometricPreconditionersBased2016}. In brief, we first compute the generalized eigendecompositions of the matrix $\widehat{\mathbf{T}}_d$ w.r.t. the scalar product induced by $\widehat{\mathbf{M}}_{d;2,2}$ for $d=1,2,3$, that is,
\begin{equation}\label{eq:gen_eigs}
	\mathbf{Q}_d^{\top} \widehat{\mathbf{T}}_d \mathbf{Q}_d = \boldsymbol{\Lambda}_d, \qquad \mathbf{Q}_d^{\top} \widehat{\mathbf{M}}_{d;2,2} \mathbf{Q}_d = \mathbf{I},
\end{equation}
where the matrix $\boldsymbol{\Lambda}_d$ is diagonal. The existence of this factorization is guaranteed, since $\widehat{\mathbf{T}}_d$ is symmetric and $\widehat{\mathbf{M}}_{d;2,2}$ is symmetric positive definite.
Subsequently, the inverse of the Schur complement $\widehat{\mathbf{S}}$ can be expressed as
\begin{equation*}
	\widehat{\mathbf{S}}^{-1} = \left( \mathbf{Q}_3 \otimes \mathbf{Q}_2 \otimes \mathbf{Q}_1 \right) \left( \mathbf{I} \otimes \mathbf{I} \otimes \mathbf{I} - \boldsymbol{\Lambda}_3 \otimes \boldsymbol{\Lambda}_2 \otimes \boldsymbol{\Lambda}_1 \right)^{-1} \left( \mathbf{Q}^{\top}_3 \otimes \mathbf{Q}^{\top}_2 \otimes \mathbf{Q}^{\top}_1 \right).
\end{equation*}
where we emphasize that the central matrix of the right-hand side is diagonal.

In summary, the procedure for solving a system associated with $\overline{\mathbf{P}}_{\overline{\mathcal{C}}}$ is outlined in Algorithm~\ref{alg:Pbk}, while the procedure for solving a system of the form $\widehat{\mathbf{S}}\mathbf{y}=\mathbf{z}$ is described in Algorithm~\ref{alg:fast_diag}.
\begin{algorithm}[ht]
	\caption{Solution of the block Kronecker preconditioner system $\overline{\mathbf{P}}_{\widetilde{\mathcal{C}}} \mathbf{y} = \mathbf{z}$}\label{alg:Pbk}
	\begin{algorithmic}[1]
		\State Compute the approximations $\boldsymbol{\omega}^{(d)}_{k,\ell}$ and assemble the matrices $\widehat{\mathbf{M}}_{d;k,\ell}$, for $d=1,2,3$ and $(k,\ell)=(1,1),(1,2),(2,2),(3,3)$. Compute the diagonal scaling $\mathbf{D}_{\widetilde{\mathfrak{C}}}$.
		\State Compute $\widetilde{\mathbf{z}} = \mathbf{D}^{-1}_{\widetilde{\mathcal{C}}} \mathbf{z}$.
		\State Compute $\widetilde{\mathbf{r}} = \mathbf{L} \widetilde{\mathbf{z}}$.
		\State Solve $\mathbf{B} \mathbf{r} = \widetilde{\mathbf{r}}$.
		\State  Compute $\widetilde{\mathbf{y}} = \mathbf{U} \mathbf{r}$.
		\State Compute $\mathbf{y} = \mathbf{D}_{\widetilde{\mathcal{C}}}^{-1} \widetilde{\mathbf{y}}$.
	\end{algorithmic}
\end{algorithm}
\begin{algorithm}[ht]
	\caption{Fast Diagonalization method to solve the Schur system $\widehat{\mathbf{S}}\mathbf{y}=\mathbf{z}$}\label{alg:fast_diag}
	\begin{algorithmic}[1]
		\State Compute the generalized eigendecompositions~\eqref{eq:gen_eigs}. 
		\State Compute $\widetilde{\mathbf{z}} = \left( \mathbf{Q}^{\top}_3 \otimes \mathbf{Q}^{\top}_2 \otimes \mathbf{Q}^{\top}_1 \right)\mathbf{z}$.
		\State Solve $ \left( \mathbf{I} \otimes \mathbf{I} \otimes \mathbf{I} - \boldsymbol{\Lambda}_3 \otimes \boldsymbol{\Lambda}_2 \otimes \boldsymbol{\Lambda}_1 \right) \widetilde{\mathbf{y}} = \widetilde{\mathbf{z}}$.
		\State Compute $\mathbf{y} = \left( \mathbf{Q}_3 \otimes \mathbf{Q}_2 \otimes \mathbf{Q}_1 \right)\ \widetilde{\mathbf{y}}$.
	\end{algorithmic}
\end{algorithm}
\subsection{Symmetric discontinuous Galerkin preconditioner}
Since the multi-patch discontinuous Galerkin formulation employs functions that are defined independently on each patch without imposing any continuity constraint on the interfaces, the mass matrix $\mathbf{M}_{\epsilon}$ exhibits a block diagonal structure, where each block corresponds to the mass matrix of a single patch, i.e.
\begin{equation*}
	{\mathbf{M}}_{\epsilon}=
	\begin{bmatrix}
		{\mathbf{M}}_{\epsilon}^{1} &        & \\
		& \ddots & \\
		&        & {\mathbf{M}}_{\epsilon}^{\nptc}
	\end{bmatrix}.
\end{equation*}
A natural choice for a preconditioner for ${\mathbf{M}}_{\epsilon}$ is given by
\begin{equation*}
	\pdg:=
	\begin{bmatrix}
		\mathbf{P}_{1} &        & \\
		& \ddots & \\
		&        & \mathbf{P}_{\nptc}
	\end{bmatrix},
\end{equation*}
where, for $\ptca=1,\ldots,\nptc$, each $\mathbf{P}_{\ptca}$ represents one of the single patch preconditioners introduced in Section~\ref{sec:single_patch_prec}, corresponding to $\Omega_{\ptca}$. Consequently, the robustness established in Theorem~\ref{thm:spectral_equiv} and Remark~\ref{rem:conditioning} extends naturally.
\subsection{Continuous Galerkin multi-patch preconditioner}\label{sec:cont_gal_prec}
To address multi-patch domains, we employ an \emph{additive Schwarz} method, as described in~\cite[Chapter~2]{toselli2006domain}, combined with the single patch preconditioners introduced in Section~\ref{sec:single_patch_prec}, following an approach similar to that in~\cite{LOLI2022245}.  
For notational consistency with~\cite{toselli2006domain}, we introduce the bilinear form associated with the mass matrix for $X_{h,0}^1(\Omega)$:
\begin{equation*}
	a : X_{h,0}^1 \times X_{h,0}^1 \rightarrow \mathbb{R},
	\qquad
	a(u,v) := (u,v)_{L^2_{\epsilon}\Lambda^1(\Omega)}
	= \mathbf{v}^{\top} \mathbf{M}_\epsilon \mathbf{u},
\end{equation*}
where $\mathbf{u}$ and $\mathbf{v}$ denote the coordinate vectors of $u$ and $v$ with respect to the basis $\mathcal{B}_{\mathrm{mp}}$.\\
We next introduce a family $\{ V_{\ptca} \}_{\ptca=1,\ldots,\nptc}$ of subspaces of $X_{h,0}^1$ such that each $V_{\ptca}$ is the span of the basis functions that do not vanish on the patch $\Omega_{\ptca}$, that is,
\begin{equation*}
	V_{\ptca} = \mathrm{span}(\mathcal{B}_{\ptca}),
	\qquad
	\mathcal{B}_{\ptca}
	:= \left\{ \beta_i \in \mathcal{B}_{\mathrm{mp}} :
	\ \mathrm{supp}(\beta_i) \cap \Omega_{\ptca} \neq \emptyset \right\}.
\end{equation*}
The support of basis functions in $\mathcal{B}_{\ptca}$ is contained in $\Omega_{\ptca}$ and in the patches that share at least one edge with it.  
For each patch $\Omega_{\ptca}$, we therefore introduce the set of \emph{adjacent patches}
\begin{equation*}
	\mathcal{N}_{\ptca}
	:= \left\{ \ptcb \in \{1,\ldots,\nptc\} :
	\ \overline{\Omega}_{\ptca} \cap \overline{\Omega}_{\ptcb}
	\ \text{contains at least one edge} \right\}
\end{equation*}
and the maximal adjacency number
\begin{equation*}
	\nadj := \max_{\ptca=1,\ldots,\nptc} |\mathcal{N}_{\ptca}|.
\end{equation*}
Because the supports of the basis functions extend into adjacent patches, two distinct subspaces $V_{\ptca}$ and $V_{\ptcb}$ can have an interaction even if $\Omega_{\ptca}$ and $\Omega_{\ptcb}$ do not share an edge. We define the set of \emph{interacting patches} for $\Omega_{\ptca}$ based on the overlap of basis function supports:
\begin{equation*}
	\mathcal{N}_{\mathrm{int}, \ptca}
	:= \left\{ \ptcb \in \{ 1,\ldots,\nptc \} :
		\ \exists \beta_i \in \mathcal{B}_{\ptca} , \ \beta_j \in \mathcal{B}_{\ptcb}
		\ \text{s.t.} \ \supp(\beta_i) \cap \supp(\beta_j) \neq \emptyset \right\}.
\end{equation*}
The actual \emph{maximal interaction number} characterizing the sparsity of the subspace decomposition is then given by:
\begin{equation*}
	\nint := \max_{\ptca=1,\ldots,\nptc} |\mathcal{N}_{\mathrm{int}, \ptca}|.
\end{equation*}
\begin{remark}	
	The indices of interacting patches $\mathcal{N}_{\mathrm{int}, \ptca}$ depend on the patch adjacency graph, but also on the element discretization. 
	Consider, for example, three aligned patches $\Omega_1$, $\Omega_2$, and $\Omega_3$, where the first and last patches are not adjacent. 
	If the mesh is very coarse (only one element along the direction of alignment) there exist functions in $\mathcal{B}_{1}$ and $\mathcal{B}_{3}$ such that their supports intersection is not empty. However, if the mesh is refined along the direction of alignment, functions in $\mathcal{B}_{1}$ and $\mathcal{B}_{3}$ would not interact anymore.
\end{remark}
For each subspace $V_{\ptca}$, we introduce a local bilinear form, called \emph{local solver}, associated with the $L^2_{\epsilon}\Lambda^1(\Omega_{\ptca})$ inner product applied to the restrictions to $\Omega_{\ptca}$ of functions in $V_{\ptca}$:
\begin{equation}\label{eq:local_solver}
	\widetilde{a}_{\ptca} : V_{\ptca} \times V_{\ptca} \rightarrow \mathbb{R},
	\qquad 
	\widetilde{a}_{\ptca}(u_{\ptca}, v_{\ptca})	:= (u_{\ptca}|_{\Omega_{\ptca}},v_{\ptca}|_{\Omega_{\ptca}})_{L^2_{\epsilon}\Lambda^1(\Omega_{\ptca})}
	= \mathbf{v}^{\top}_{\ptca} \mathbf{G}^{\top}_{\ptca} 	\mathbf{M}^{\ptca}_{\epsilon} 	\mathbf{G}_{\ptca} \mathbf{u}_{\ptca},
\end{equation}
where $\mathbf{u}_{\ptca}$ and $\mathbf{v}_{\ptca}$ are the coordinate vectors of $u_{\ptca}$ and $v_{\ptca}$ relative to the basis $\mathcal{B}_{\ptca}$, $\mathbf{M}^{\ptca}_{\epsilon}$ is the mass matrix on $\Omega_{\ptca}$, and $\mathbf{G}_{\ptca}$ is a diagonal matrix with entries $\pm 1$ that implements the orientation change from the multi-patch basis to the single patch basis.\\
We define the ideal multi-patch additive Schwarz preconditioner as
\begin{equation*}
	\pad^{-1} := \sum_{\ptca=1}^{\nptc} \mathbf{R}_{\ptca}^{\top} \mathbf{G}^{\top}_{\ptca} \left(\mathbf{M}^{\ptca}_{\epsilon}\right)^{-1} \mathbf{G}_{\ptca} \mathbf{R}_{\ptca},
\end{equation*}
where $\mathbf{R}_{\ptca} \in \mathbb{R}^{N_{\mathrm{dof},\ptca} \times \ndof}$ extracts the degrees of freedom associated with the local space $V_{\ptca}$, and $N_{\mathrm{dof},\ptca}$ is the dimension of this local space.  
The next result provides an upper bound for the condition number of the matrix $\mathbf{M}_\epsilon$, symmetrically preconditioned with $\pad$.
\begin{theorem}\label{thm:additive_schwarz}
	There exists a positive constant $C$, independent of the meshes defined by $\mathbf{\Xi}^{(\ptca)}$ for $\ptca = 1,\ldots,\nptc$, the total number of patches $\nptc$, and the parameters $\nadj$ and $\nint$, but possibly dependent on the spline degrees $\mathbf{p}$, the geometric mapping $\mathbf{F}$, and the physical parameter $\epsilon$, such that
	\begin{equation*}
		\kappa_2 \left( \pad^{-1/2} \mathbf{M}_\epsilon \, \pad^{-1/2} \right) \leq C \nadj (\nint +1).
	\end{equation*}
\end{theorem}
\begin{proof}
If the \emph{stable decomposition}, \emph{strengthened Cauchy-Schwarz inequalities}, and \emph{local stability} conditions hold, see~\cite[Assumptions 2.2, 2.3, and 2.4]{toselli2006domain}, then the conclusion follows directly from~\cite[Theorem 2.7]{toselli2006domain}.
\begin{itemize}
\item \textbf{Stable decomposition.}
We have to prove that there exists $C_0>0$, independent of the meshes $\mathbf{\Xi}^{(\ptca)}$ but possibly dependent on $\mathbf{p}$, $\mathbf{F}$, and $\epsilon$, such that every 1-form $u \in X_{h,0}^1$ admits a decomposition
\begin{equation*}
u = \sum_{\ptca=1}^{\nptc} u^{(\ptca)}, \qquad u^{(\ptca)} \in V_{\ptca},
\end{equation*}
satisfying
\begin{equation}\label{eq:stab_dec}
	\sum_{\ptca=1}^{\nptc} \widetilde{a}_{\ptca}(u^{(\ptca)},u^{(\ptca)}) \leq C_0^2\, a(u,u).
\end{equation}
Any $u \in X_{h,0}^1$ can be written as
\begin{equation*}
	u = \sum_{i=1}^{\ndof} u_i \beta_i, \quad \beta_i \in \mathcal{B}_{\mathrm{mp}}.
\end{equation*}
For each patch $\ptca=1,\ldots,\nptc$, we define
\begin{equation*}
	u^{(\ptca)} := \sum_{i=1}^{\ndof} u_i^{(\ptca)} \beta_i,
	\qquad
	u_i^{(\ptca)} :=
	\begin{cases}
		0, & \text{if } \supp(\beta_i) \cap \Omega_{\ptca} = \emptyset,\\[1mm]
		\frac{u_i}{n_i}, & \text{otherwise},
	\end{cases}
\end{equation*}
where 
\begin{equation*}
	n_i := \#\{ \ptcb : \supp(\beta_i) \cap \Omega_{\ptcb} \neq \emptyset \}.
\end{equation*}
By construction, $u = \sum_{\ptca=1}^{\nptc} u^{(\ptca)}$.
The global $L^2_{\epsilon}\Lambda^1$-norm can be written as a sum over the patches:
\begin{equation*}
	a(u,u) = (u,u)_{L^2_{\epsilon}\Lambda^1(\Omega)} = \sum_{\ptca=1}^{\nptc} (u|_{\Omega_{\ptca}}, u|_{\Omega_{\ptca}})_{L^2_{\epsilon}\Lambda^1(\Omega_{\ptca})}.
\end{equation*}
Since the restriction of each $\beta_i$ to $\Omega_{\ptca}$ coincides, up to a sign, with an element of the local patch basis $\mathcal{B}_{\ptca}^1$, we have
\begin{equation*}
	(u|_{\Omega_{\ptca}}, u|_{\Omega_{\ptca}})_{L^2_{\epsilon}\Lambda^1(\Omega_{\ptca})}
	= \mathbf{u}^{\top} \mathbf{R}_{\ptca}^{\top} \mathbf{G}_{\ptca}^{\top} \mathbf{M}^{\ptca}_\epsilon \mathbf{G}_{\ptca} \mathbf{R}_{\ptca} \mathbf{u},
\end{equation*}
where $\mathbf{u}$ is the coordinate vector of $u$ with respect to $\mathcal{B}_{\mathrm{mp}}$.\\
For the local component $u^{(\ptca)}$, we similarly have
\begin{equation*}
	\widetilde{a}_{\ptca}(u^{(\ptca)}, u^{(\ptca)})
	= (u^{(\ptca)}|_{\Omega_{\ptca}}, u^{(\ptca)}|_{\Omega_{\ptca}})_{L^2_{\epsilon}\Lambda^1(\Omega_{\ptca})}
	= \mathbf{u}^{\top} \mathbf{E}^{\top} \mathbf{R}_{\ptca}^{\top} \mathbf{G}_{\ptca}^{\top} \mathbf{M}^{\ptca}_\epsilon \mathbf{G}_{\ptca} \mathbf{R}_{\ptca} \mathbf{E} \mathbf{u},
\end{equation*}
where $\mathbf{E} := \mathrm{diag}(1/n_1, \dots, 1/n_{\ndof})$. 

Recalling that the restriction matrix $\mathbf{R}_{\ptca}$ selects the degrees of freedom corresponding to $V_\ptca$, we define the local vector $\mathbf{u}_{\ptca} = [u_{i_1}, \dots, u_{i_{N_{\mathrm{dof},\ptca}}}]^\top$ as
\begin{equation*}
	\mathbf{u}_{\ptca} := \mathbf{R}_{\ptca} \mathbf{u}.
\end{equation*}
Next, we introduce the local diagonal matrix $\mathbf{E}_{\ptca} \in \mathbb{R}^{N_{\mathrm{dof},\ptca} \times N_{\mathrm{dof},\ptca}}$, given by
\begin{equation*}
	\mathbf{E}_{\ptca} := \mathrm{diag}(1/n_{i_1}, \dots, 1/n_{i_{N_{\mathrm{dof},\ptca}}}).
\end{equation*}
This allows us to write
\begin{equation*}
	\mathbf{R}_{\ptca} \mathbf{E} \mathbf{u} = \mathbf{E}_{\ptca} \mathbf{u}_{\ptca}.
\end{equation*}
Consequently, we obtain the representations
\begin{equation*}
	\begin{aligned}
		(u|_{\Omega_{\ptca}}, u|_{\Omega_{\ptca}})_{L^2_{\epsilon}\Lambda^1(\Omega_{\ptca})}
		&= \mathbf{u}_{\ptca}^{\top} \mathbf{G}_{\ptca}^{\top} \mathbf{M}^{\ptca}_\epsilon \mathbf{G}_{\ptca} \mathbf{u}_{\ptca},\\
		\widetilde{a}_{\ptca}(u^{(\ptca)}, u^{(\ptca)})
		&= \mathbf{u}_{\ptca}^{\top} \mathbf{E}_{\ptca}^{\top} \mathbf{G}_{\ptca}^{\top} \mathbf{M}^{\ptca}_\epsilon \mathbf{G}_{\ptca} \mathbf{E}_{\ptca} \mathbf{u}_{\ptca}.
	\end{aligned}
\end{equation*}
Since $\mathbf{G}_{\ptca}$ is diagonal with entries equal to $\pm 1$, it follows that
\begin{equation*}
	\frac{\mathbf{v}_{\ptca}^{\top} \mathbf{E}_{\ptca}^{\top} \mathbf{G}_{\ptca}^{\top} \mathbf{M}^{\ptca}_\epsilon \mathbf{G}_{\ptca} \mathbf{E}_{\ptca} \mathbf{v}_{\ptca}}
	{\mathbf{v}_{\ptca}^{\top} \mathbf{G}_{\ptca}^{\top} \mathbf{M}^{\ptca}_\epsilon \mathbf{G}_{\ptca} \mathbf{v}_{\ptca}}
	\leq \frac{\lambda_{\max}(\mathbf{E}_{\ptca})^2 \lambda_{\max}(\mathbf{G}_{\ptca})^2 \lambda_{\max}(\mathbf{M}^{\ptca}_\epsilon)}
	{\lambda_{\min}(\mathbf{G}_{\ptca})^2 \lambda_{\min}(\mathbf{M}^{\ptca}_\epsilon)}
	\leq \kappa_2(\mathbf{M}^{\ptca}_\epsilon).
\end{equation*}
Using arguments analogous to Theorem~\ref{thm:spectral_equiv}, we have
\begin{equation*}
	\lambda_{\max}(\mathbf{M}^{\ptca}_\epsilon) \leq \sup_{\widehat{\mathbf{x}}\in\widehat{\Omega}} \lambda_{\max}(\mathcal{C}^{(\ptca)}(\widehat{\mathbf{x}})) \max_{k=1,2,3} \lambda_{\max}(\widehat{\mathbf{M}}_{k,k}), 
	\quad
	\lambda_{\min}(\mathbf{M}^{\ptca}_\epsilon) \geq \inf_{\widehat{\mathbf{x}}\in\widehat{\Omega}} \lambda_{\min}(\mathcal{C}^{(\ptca)}(\widehat{\mathbf{x}})) \min_{k=1,2,3} \lambda_{\min}(\widehat{\mathbf{M}}_{k,k}).
\end{equation*}
Applying~\cite[Lemma~1]{LOLI2022245}, we conclude that
\begin{equation*}
	C_0^2 := \max_{\ptca=1,\dots,\nptc}  \left(\sup_{\widehat{\mathbf{x}}\in\widehat{\Omega}} \kappa_2(\mathcal{C}^{(\ptca)}(\widehat{\mathbf{x}}))\right) \max_{k=1,2,3} \kappa_2(\widehat{\mathbf{M}}_{k,k})
\end{equation*}
is a positive constant independent of the meshes $\mathbf{\Xi}^{(\ptca)}$, but possibly dependent on $\mathbf{p}$, $\mathbf{F}$, and $\epsilon$, so that inequality~\eqref{eq:stab_dec} holds.
\item \textbf{Strengthened Cauchy-Schwarz inequalities.}
We have to prove that for any $\ptca,\ptcb = 1,\ldots,\nptc$ there exists $\varepsilon_{\ptca,\ptcb} \in [0,1]$ such that
\begin{equation*}
	a(v^{(\ptca)},v^{(\ptcb)}) \leq \varepsilon_{\ptca,\ptcb} a(v^{(\ptca)},v^{(\ptca)})^{1/2} a(v^{(\ptcb)},v^{(\ptcb)})^{1/2},
	\qquad
	\forall\, v^{(\ptca)} \in V_{\ptca},\ v^{(\ptcb)} \in V_{\ptcb}.
\end{equation*}
We define
\begin{equation*}
	\varepsilon_{\ptca,\ptcb} :=
	\begin{cases}
		1, & \text{if } \ptcb \in \mathcal{N}_{\mathrm{int}, \ptca}, \\
		0, & \text{otherwise (non-interacting patches)}.
	\end{cases}
\end{equation*}
For interacting patches, the estimate follows from the Cauchy--Schwarz inequality. For non-interacting patches, the supports are disjoint and the bilinear form vanishes.\\
Let $\mathcal{E} = \{\varepsilon_{i,j}\}$ denote the resulting interaction matrix.  
Since each patch interacts with at most $\nint$ patches, every row of $\mathcal{E}$ contains at most $\nint$ nonzero entries, all equal to~1, and hence
\begin{equation*}
	\rho(\mathcal{E}) \leq \nint.
\end{equation*}
\item \textbf{Local stability.}
We have to prove that there exists a constant $\omega>0$ such that, for every $\ptca=1,\ldots,\nptc$ and for every $u^{(\ptca)} \in V_{\ptca}$,
\begin{equation} \label{eq:local_stability}
	a(u^{(\ptca)},u^{(\ptca)}) \leq \omega\, \widetilde{a}_{\ptca}(u^{(\ptca)},u^{(\ptca)}).
\end{equation}
For any $u^{(\ptca)} \in V_{\ptca}$, the support of $u^{(\ptca)}$ is contained in the patch $\Omega_{\ptca}$ and the neighboring patches indexed by $\mathcal{N}_{\ptca}$. Therefore,
\begin{equation*}
	a(u^{(\ptca)},u^{(\ptca)}) = \|u^{(\ptca)}\|^{2}_{L^{2}_{\epsilon}\Lambda^{1}(\Omega)}
	= \sum_{\ptcb \in \mathcal{N}_{\ptca}} \left\| u^{(\ptca)}\big|_{\Omega_{\ptcb}} \right\|^{2}_{L^{2}_{\epsilon}\Lambda^{1}(\Omega_{\ptcb})}.
\end{equation*}
Let $\mathbf{u}$ denote the coordinate vector of $u^{(\ptca)}|_{\Omega_{\ptca}}$ with respect to the single patch 1-form basis $\mathcal{B}^1$ defined on $\Omega_{\ptca}$. Then
\begin{equation*}
\widetilde{a}_{\ptca}(u^{(\ptca)},u^{(\ptca)}) = \left\| u^{(\ptca)}\big|_{\Omega_{\ptca}} \right\|^{2}_{L^{2}_{\epsilon}\Lambda^{1}(\Omega_{\ptca})} = \mathbf{u}^{\top}\,\mathbf{M}^{\ptca}_{\epsilon}\,\mathbf{u}.
\end{equation*}
For an adjacent patch $\Omega_{\ptcb}$, the $L^{2}_{\epsilon}\Lambda^{1}(\Omega_{\ptcb})$-norm of the restriction $u^{(\ptca)}|_{\Omega_{\ptcb}}$ can be represented as
\begin{equation*}
	\left\| u^{(\ptca)}\big|_{\Omega_{\ptcb}} \right\|^{2}_{L^{2}_{\epsilon}\Lambda^{1}(\Omega_{\ptcb})}
	= \mathbf{u}^{\top}\, \mathbf{J}_{\ptcb,\ptca}^{\top}\, \mathbf{M}^{\ptcb}_{\epsilon}\, \mathbf{J}_{\ptcb,\ptca}\, \mathbf{u},
\end{equation*}
where $\mathbf{J}_{\ptcb,\ptca} \in \mathbb{R}^{N_{\mathrm{dof},\ptcb} \times N_{\mathrm{dof},\ptca}}$ is a rectangular matrix whose only nonzero entries (equal to $\pm1$) appear in the rows and columns corresponding to the degrees of freedom on the interface between $\Omega_{\ptca}$ and $\Omega_{\ptcb}$, with at most one such entry per row and column.
Consequently, $\mathbf{J}_{\ptcb,\ptca}^{\top}\mathbf{J}_{\ptcb,\ptca}$ is diagonal with diagonal entries equal to $0$ or $\pm 1$, and thus
\begin{equation*}
	\lambda_{\max}\left(\mathbf{J}_{\ptcb,\ptca}^{\top}\mathbf{J}_{\ptcb,\ptca}\right) \leq 1.
\end{equation*}
Applying~\cite[Lemma~1]{LOLI2022245} to both patches $\Omega_{\ptca}$ and $\Omega_{\ptcb}$, we obtain
\begin{align*}
	\mathbf{u}^{\top}\mathbf{J}_{\ptcb}^{\top}\, \mathbf{M}^{\ptcb}_{\epsilon}\, \mathbf{J}_{\ptcb}\mathbf{u}
	&\leq \lambda_{\max} \left(\mathbf{J}_{\ptcb}^{\top}\mathbf{J}_{\ptcb}\right)\,
	\lambda_{\max}\left(\mathbf{M}^{\ptcb}_{\epsilon}\right)\,
	\|\mathbf{u}\|^{2}_{2} \\
	&\leq \frac{\lambda_{\max}\left(\mathbf{M}^{\ptcb}_{\epsilon}\right)}
	{\lambda_{\min}\left(\mathbf{M}^{\ptca}_{\epsilon}\right)}
	\mathbf{u}^{\top}\mathbf{M}^{\ptca}_{\epsilon}\mathbf{u} \leq C_{\ptca,\ptcb}\, \mathbf{u}^{\top}\mathbf{M}^{\ptca}_{\epsilon}\mathbf{u},
\end{align*}
for a constant $C_{\ptca,\ptcb}>0$ that is independent of the mesh $\mathbf{\Xi}$, but may depend on $\mathbf{p}$, $\mathbf{F}$, and $\epsilon$.

Summing over all adjacent patches $\ptcb \in \mathcal{N}_{\ptca}$, we obtain
\begin{equation*}
	a(u^{(\ptca)},u^{(\ptca)}) \leq \left( \sum_{\ptcb \in \mathcal{N}_{\ptca}} C_{\ptca,\ptcb} \right)
	\mathbf{u}^{\top}\mathbf{M}^{\ptca}_{\epsilon}\mathbf{u}.
\end{equation*}
Finally, using the definition
\begin{equation}\label{eq:omega}
	\omega := \nadj \max_{\ptca,\ptcb=1,\ldots,\nptc} C_{\ptca,\ptcb},
\end{equation}
we conclude that inequality~\eqref{eq:local_stability} holds for all $u^{(\ptca)} \in V_{\ptca}$.
\end{itemize}
Finally,~\cite[Theorem 2.7]{toselli2006domain} implies
\begin{equation*}
	\kappa_2 \left( \pad^{-\frac{1}{2}}{\mathbf M}_\epsilon \pad^{-\frac{1}{2}} \right) \leq \omega C_0^2 (\rho(\mathcal{E})+1) \leq C \nadj (\nint +1),
\end{equation*}
where we have set $\displaystyle C=C_0^2 \max_{\ptca,\ptcb=1,\ldots,\nptc} C_{\ptca,\ptcb}$.
\end{proof}
\begin{remark}\label{rem:additive_schwarz_performance}
As in the single-patch case, the preceding result implies that the number of iterations required for the conjugate gradient method to converge, when preconditioned with $\pad$, is independent of the mesh size. However, unlike the single-patch setting, the iteration count in the continuous Galerkin multi-patch formulation may depend on the spline degree. Furthermore, the proposed additive Schwarz preconditioner is \emph{scalable}: the condition number of the preconditioned operator does not depend on the total number of patches, but solely on the maximum number of adjacent patches.
\end{remark}
The results stated in Theorem~\ref{thm:additive_schwarz} and Remark~\ref{rem:additive_schwarz_performance} remain valid if the local solver ${\mathbf{M}}^{\ptca}_\epsilon$ is replaced by any spectrally equivalent preconditioner, such as $\overline{\mathbf{P}}_{\mathbf{I}}$ or $\overline{\mathbf{P}}_{\overline{\mathcal{C}}}$, in the definition of $\widetilde{a}_{\ptca}$ in~\eqref{eq:local_solver}. This follows directly from the spectral equivalence established in Theorem~\ref{thm:spectral_equiv}, which guarantees that the properties of \emph{stable decomposition} and \emph{local stability} remain valid.  
\subsection{Computational complexity}
The mass matrix ${\mathbf{M}}_{\epsilon}$, as well as all the preconditioners considered in this paper, is symmetric and positive definite. Therefore, the PCG method can be effectively applied to solve the resulting linear systems. We remark that both the additive Schwarz preconditioner for the continuous Galerkin formulation and the block diagonal preconditioner for the symmetric discontinuous Galerkin formulation apply single patch preconditioners independently to each patch. This structure supports a highly parallelizable implementation, enhancing computational efficiency.

We now analyze the computational costs associated with the single patch preconditioners $\overline{\mathbf{P}}_{\mathbf{I}}$ and $\overline{\mathbf{P}}_{\overline{\mathcal{C}}}$ described in Section~\ref{sec:single_patch_prec}. We distinguish between the \emph{setup} operations, which are executed only once before the PCG solver begins, and the \emph{application} operations, which are performed at each PCG iteration.
The costs of these operations should be compared with that of the residual computation that is performed at each iteration of PCG. 
For simplicity, in the following we consider the same degree $p$ and mesh size in all directions. Each column of the matrix ${\mathbf{M}}_{\epsilon}$, whose order is denoted with $\ndof$, contains at most $3 (2 p + 1)^2 (2p-1)$ nonzero entries, thus the number of floating-point operations (FLOPs) required to compute the residual is $\mathcal{O}(p^3 \ndof)$.

The setup of both $\overline{\mathbf{P}}_{\mathbf{I}}$ and $\overline{\mathbf{P}}_{\overline{\mathcal{C}}}$ requires assembling the mass matrices for the univariate spaces and computing the diagonal scaling. If we assume, here and in the remainder of this section, that the number of degrees of freedom is roughly the same in each direction, the combined cost of these steps is $\mathcal{O}(p^3 \ndof^{1/3} + \ndof)$ FLOPs. 
Furthermore, the setup of $\overline{\mathbf{P}}_{\overline{\mathcal{C}}}$ requires computing the separate variable approximations~\eqref{eq:separable_factorization}, whose cost is proportional to the number of quadrature nodes, i.e., $\mathcal{O}(p^3 \ndof)$ FLOPs, as well as the generalized eigendecompositions~\eqref{eq:gen_eigs} in Algorithm~\ref{alg:fast_diag}, whose cost is $\mathcal{O}(\ndof)$ FLOPs.

Regarding the application operations, both  $\overline{\mathbf{P}}_{\mathbf{I}}$ and $\overline{\mathbf{P}}_{\overline{\mathcal{C}}}$ require the solution of few linear systems with Kronecker matrices, whose Kronecker factors are banded matrices of order roughly $\frac{1}{9}\ndof^{1/3}$ and bandwidth $p$. This step has a computational cost of 
$\mathcal{O}(p \ndof)$ FLOPs, see e.g.~\cite{LOLI2022245} for details. Furthermore, the application of $\overline{\mathbf{P}}_{\overline{\mathcal{C}}}$ requires to compute matrix-vector products with the off-diagonal block $\widehat{\mathbf{M}}_{1,2}$ and its transpose, whose cost is analogous to the previous one, as well as with the dense eigenvector matrix $\mathbf{Q}_3 \otimes \mathbf{Q}_2 \otimes \mathbf{Q}_1$ and its transpose, in the application of Algorithm~\ref{alg:fast_diag}. Although the theoretical cost of the latter operation is $\mathcal{O}(\ndof^{4/3})$, previous observations in~\cite{loli20202586} and our numerical experiments in Section~\ref{sec:num} indicate that this operation is very fast in practice.

The overall computational complexities of the proposed single patch preconditioners are summarized in Table~\ref{tab:compCost}.
\begin{table}[htbp]
\centering
\begin{tabular}{rcc}
	\toprule
	& $\overline{\mathbf{P}}_{\mathbf{I}}$ & $\overline{\mathbf{P}}_{\overline{\mathcal{C}}}$ \\
	\midrule
	Setup & $\mathcal{O}(p^3 \ndof^{1/3} + \ndof)$ & $\mathcal{O}(p^3 \ndof)$ \\
	Application & $\mathcal{O}(p \ndof)$ & $\mathcal{O}(\ndof^{4/3})$ \\
	\bottomrule
\end{tabular}
\caption{Summary of the computational complexity of the proposed single patch preconditioner.}
\label{tab:compCost}
\end{table}
\begin{remark}
In a time-marching scheme, at each time step, we need to solve a linear system associated with the same matrices, $\mathbf{M}_{\epsilon}$ and $\widetilde{\mathbf{M}}_{\mu}$. Consequently, the same preconditioner can be used throughout the entire process. This implies that the setup procedures for the preconditioner only need to be performed once at the beginning of the time-marching scheme. As a result, the computational cost of these setup procedures becomes negligible compared to the application procedures, which must be executed at every time step and every PCG iteration.
\end{remark}
\section{Numerical results}\label{sec:num}
In this section, we present numerical experiments to assess the robustness and computational efficiency of the proposed preconditioners. Specifically, we evaluate the number of iterations required for the PCG method to converge when solving the linear systems associated with $\mathbf{M}_{\epsilon}$, setting $\epsilon = 1$. The right-hand side is randomly generated, and the initial guess is set to the zero vector.

Throughout this study, we assume uniform spline degrees across all patches, i.e., $p_1 = p_2 = p_3 = p$. Mesh refinement is performed by uniformly subdividing the three knot vectors that define the geometric spline parametrization $\mathbf{F}_{\ptca}$, resulting in $n_{\text{sub}}$ subdivisions in each parametric direction.

All numerical experiments were conducted on a computing cluster equipped with two Intel Xeon 16-Core 6130 CPUs (2.1 GHz, 22 MB cache) and 128~GB of RAM. The simulations were carried out using \textsc{MATLAB R2023a}, in combination with the GeoPDEs toolbox~\cite{VAZQUEZ2016523} and the Tensorlab package~\cite{tensorlab30} for CPD and Kronecker matrix-vector products.
Algorithm~\ref{alg:sep_approx} was executed with \texttt{maxIter}~$= 3$ iterations. The PCG method was implemented using \textsc{MATLAB}'s built-in solver, with the convergence tolerance set to $\tau = 10^{-8}$.
Although the proposed algorithms are designed to support highly parallel implementations, all experiments in this section are executed in a single-threaded environment without parallel processing or multi-threading techniques. This approach was chosen to emphasize the inherent efficiency of the proposed methods.
\subsection{Example 1: Single patch preconditioners}
In this example, we compare the performance of $\overline{\mathbf{P}}_{\mathbf{I}}$ and $\overline{\mathbf{P}}_{\overline{\mathcal{C}}}$ on the single patch curved domain depicted in Figure~\ref{fig:singlepatch}.
\begin{figure}[htbp]
	\centering
	\includegraphics[width=0.49\textwidth]{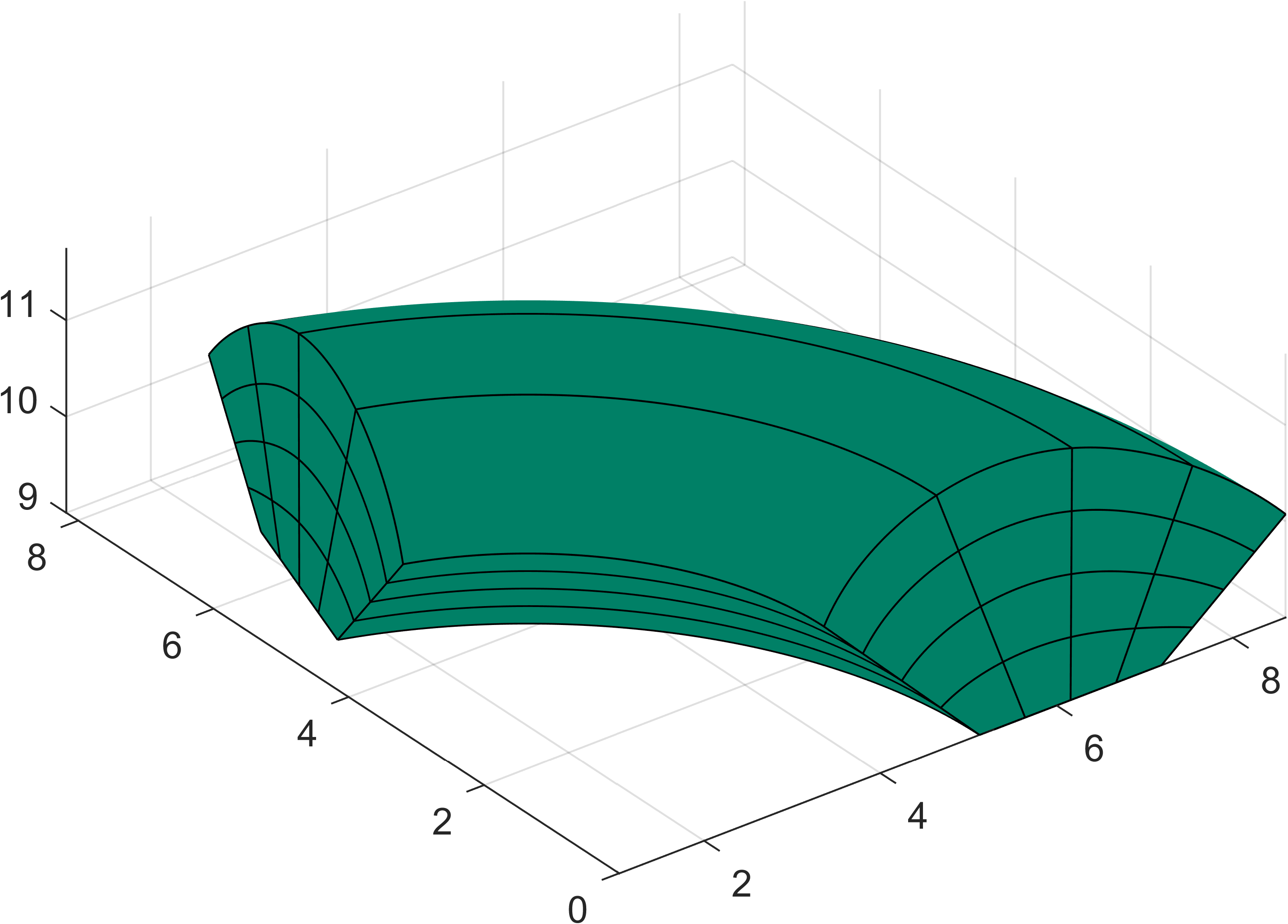}
	\caption{Single patch curved domain.
	}
	\label{fig:singlepatch}
\end{figure}

In Table~\ref{tab:ITERvolPtc2E}, we present the number of PCG iterations required to solve a system associated with $\mathbf{M}_{\epsilon}$ for various spline degrees $p$ and different numbers of subdivisions $n_{\text{sub}}$. The results indicate that both preconditioners, $\overline{\mathbf{P}}_{\mathbf{I}}$ and $\overline{\mathbf{P}}_{\overline{\mathcal{C}}}$, exhibit robustness with respect to the spline degree and the number of subdivisions. Specifically, the iteration count remains relatively stable, with the number of iterations consistently staying below $24$ for $\overline{\mathbf{P}}_{\mathbf{I}}$ and $11$ for $\overline{\mathbf{P}}_{\overline{\mathcal{C}}}$. 
Note that, on the finest discretization level, $\ndof \approx6.6$ millions.
As expected, $\overline{\mathbf{P}}_{\overline{\mathcal{C}}}$ outperforms $\overline{\mathbf{P}}_{\mathbf{I}}$, suggesting that $\overline{\mathcal{C}}_{k,\ell}$ provides a good approximation of $\mathcal{C}_{k,\ell}$. 
\pgfplotstableread[col sep=comma]{tables/bdPerfITERvolPtc2E.csv}\tabBD
\pgfplotstableread[col sep=comma]{tables/fdPerfITERvolPtc2E.csv}\tabBK

\begin{table}[htbp]
\centering
\pgfplotstabletypeset[columns={nsub,ItsEp2,ItsEp3,ItsEp4,ItsEp5}, 
column type=c, 
every head row/.style={output empty row,before row={\toprule \multirow{2}{*}{$n_{\text{sub}}$} & \multicolumn{4}{c}{$\overline{\mathbf{P}}_{\mathbf{I}}$}\\ \cmidrule(lr){2-5} & $p=2$ & $p=3$ & $p=4$ & $p=5$\\}, after row=\midrule},
every last row/.style={after row=\bottomrule},
columns/nsub/.style ={column name={}},
columns/ItsEp2/.style ={column name={},column type=r},
columns/ItsEp3/.style ={column name={},column type=r},
columns/ItsEp4/.style ={column name={},column type=r,string type},
columns/ItsEp5/.style ={column name={},column type=r,string type}
]\tabBD
\pgfplotstabletypeset[columns={ItsEp2,ItsEp3,ItsEp4,ItsEp5}, 
column type=c, 
every head row/.style={output empty row,before row={\toprule \multicolumn{4}{c}{$\overline{\mathbf{P}}_{\overline{\mathcal{C}}}$}\\ \cmidrule(lr){1-4} $p=2$ & $p=3$ & $p=4$ & $p=5$\\}, after row=\midrule},
every last row/.style={after row=\bottomrule},
columns/ItsEp2/.style ={column name={},column type=r},
columns/ItsEp3/.style ={column name={},column type=r},
columns/ItsEp4/.style ={column name={},column type=r,string type},
columns/ItsEp5/.style ={column name={},column type=r,string type}
]\tabBK
\caption{Example 1: PCG iterations for solving a system associated with ${\mathbf{M}}_{\epsilon}$ in the single patch curved domain, as shown in Figure~\ref{fig:singlepatch}. Entries marked with $*$ indicate cases where the problem exceeded the capacity of the available computational resources.}
\label{tab:ITERvolPtc2E}
\end{table}
Furthermore, Figure~\ref{fig:ITERvolPtc2E} reports the wall-clock time for both the setup and solution phases, considering both preconditioners. We emphasize that the application time of $\overline{\mathbf{P}}_{\overline{\mathcal{C}}}$ appears to scale as $\ndof$ for the considered discretization levels, despite the $\mathcal{O}(\ndof^{4/3})$ complexity of the fast diagonalization method (see Table~\ref{tab:compCost}). A similar behavior was already observed e.g. in \cite{sangalliIsogeometricPreconditionersBased2016}, and is probably related to the efficiency of level 3 BLAS routines that compute dense matrix-matrix products.

Moreover, in the setup time of $\overline{\mathbf{P}}_{\mathbf{I}}$, for all the considered discretization levels the $\mathcal{O}(p^3 \ndof^{1/3})$ cost to assemble the mass matrices for the univariate spaces appears to dominate over the $\mathcal{O}(\ndof)$ cost to compute the diagonal scaling. 
Apart from these, the other timings align well with the theoretical complexities presented in Table~\ref{tab:compCost}.
\tikzstyle{Linea1}=[thick,dashed]
\tikzstyle{Linea2}=[thick]
\pgfplotscreateplotcyclelist{Lista1rs}{%
	{Linea2,red,mark=*},
	{Linea1,red,mark=*},
	{Linea2,green,mark=triangle*},
	{Linea1,green,mark=triangle*},
	{Linea2,cyan,mark=square*},
	{Linea1,cyan,mark=square*},
	{Linea2,violet,mark=diamond*},
	{Linea1,violet,mark=diamond*},
}
\def \tabBD{tables/bdPerfITERvolPtc2E.csv}
\def \tabBK{tables/fdPerfITERvolPtc2E.csv}
\begin{figure}[htbp]
	\centering
	\hspace*{\fill}
	\begin{subfigure}[t]{0.4\linewidth}
		\centering
		\begin{tikzpicture}[
			trim axis left
			]
			\begin{loglogaxis}[
				cycle list name=Lista1rs,
				width=\linewidth,
				height=\linewidth,
				xlabel={$\ndof$},
				ymin=1e-3,
				ymax=7e2,
				xminorticks=false,
				yminorticks=false,
				ylabel={Setup time [s]},
				legend columns=1,
				legend style={at={(0.01,0.99)},anchor=north west},
				xmajorgrids=true,
				ymajorgrids=true,
				legend entries={$p=2$,,$p=3$,,$p=4$,,$p=5$,}
				]
				\addplot table [x=IDofsEp2, y=TsetEp2, col sep=comma]{\tabBD};
				\addplot table [x=IDofsEp2, y=TsetEp2, col sep=comma]{\tabBK};
				\addplot table [x=IDofsEp3, y=TsetEp3, col sep=comma]{\tabBD};
				\addplot table [x=IDofsEp3, y=TsetEp3, col sep=comma]{\tabBK};
				\addplot table [x=IDofsEp4, y=TsetEp4, col sep=comma]{\tabBD};			
				\addplot table [x=IDofsEp4, y=TsetEp4, col sep=comma]{\tabBK};
				\addplot table [x=IDofsEp5, y=TsetEp5, col sep=comma]{\tabBD};			
				\addplot table [x=IDofsEp5, y=TsetEp5, col sep=comma]{\tabBK};
				\logLogSlopeTriangle{0.85}{0.5}{0.1}{1/3}{black};
				\logLogSlopeTriangle{0.85}{0.2}{0.6}{1}{black};
			\end{loglogaxis}
		\end{tikzpicture}
	\end{subfigure}
	\hfill
	\begin{subfigure}[t]{0.4\linewidth}
		\centering
		\begin{tikzpicture}[
			trim axis right
			]
			\begin{loglogaxis}[
				cycle list name=Lista1rs,
				width=\linewidth,
				height=\linewidth,
				xlabel={$\ndof$},
				ymin=1e-3,
				ymax=7e2,
				xminorticks=false,
				yminorticks=false,
				ylabel={PCG time [s]},
				legend columns=1,
				legend style={at={(0.01,0.99)},anchor=north west},
				xmajorgrids=true,
				ymajorgrids=true,
				legend entries={$p=2$,,$p=3$,,$p=4$,,$p=5$,}
				]
				\addplot table [x=IDofsEp2, y=TprecEp2, col sep=comma] {\tabBD};
				\addplot table [x=IDofsEp2, y=TprecEp2, col sep=comma] {\tabBK};
				\addplot table [x=IDofsEp3, y=TprecEp3, col sep=comma] {\tabBD};
				\addplot table [x=IDofsEp3, y=TprecEp3, col sep=comma] {\tabBK};
				\addplot table [x=IDofsEp4, y=TprecEp4, col sep=comma] {\tabBD};				
				\addplot table [x=IDofsEp4, y=TprecEp4, col sep=comma] {\tabBK};
				\addplot table [x=IDofsEp5, y=TprecEp5, col sep=comma] {\tabBD};				
				\addplot table [x=IDofsEp5, y=TprecEp5, col sep=comma] {\tabBK};
				\logLogSlopeTriangle{0.85}{0.2}{0.5}{1}{black};
			\end{loglogaxis}
		\end{tikzpicture}
	\end{subfigure}
	\hspace*{\fill}
	\caption{Example 1: time required to solve a system associated with $\mathbf{M}_{\epsilon}^1$ in the single patch curved domain, as shown in Figure~\ref{fig:singlepatch}. Results are presented for $\overline{\mathbf{P}}_{\mathbf{I}}$ (solid lines \rule[0.5ex]{0.45cm}{0.5pt}) and $\overline{\mathbf{P}}_{\overline{\mathcal{\mathfrak}{C}}}$ (dashed lines \rule[0.5ex]{0.15cm}{0.5pt}~\rule[0.5ex]{0.15cm}{0.5pt}~\rule[0.5ex]{0.15cm}{0.4pt}).}
	\label{fig:ITERvolPtc2E}
\end{figure}
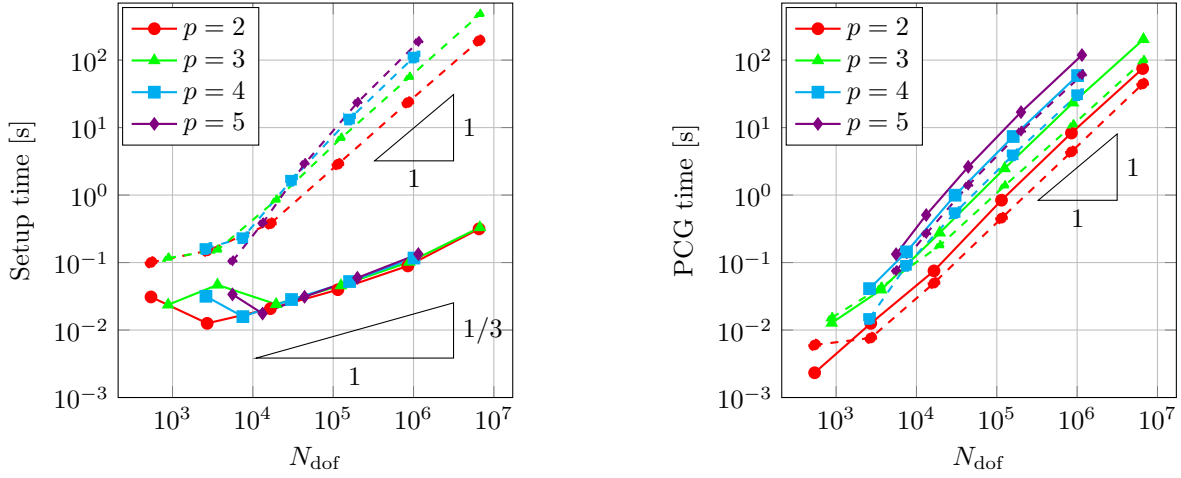
\subsection{Example 2: Continuous Galerkin multi-patch preconditioners}
To evaluate the performance differences between $\overline{\mathbf{P}}_{\mathbf{I}}$ and $\overline{\mathbf{P}}_{\overline{\mathcal{C}}}$ as local solvers within the additive Schwarz preconditioner discussed in Section~\ref{sec:cont_gal_prec}, we consider multi-patch domains inspired by the ITER tokamak, as illustrated in Figure~\ref{fig:ITER_mp}. The 2D cross-section used to generate this domain was extracted from the 3D-printable CAD files provided by the ITER Organization \cite{iter_tokamak_3d_printing}. Specifically, to analyze the scalability of the preconditioners, we examine four different geometries constructed by successively combining 2, 3, 4, and 5 sectors. Each sector is generated by rotating the two-dimensional multi-patch cross-section shown in Figure~\ref{fig:iterSec} through an angle of $2\pi/5$ radians.
\begin{figure}[htbp]
	\centering
	\begin{subfigure}{0.2\textwidth}
		\includegraphics[width=\textwidth]{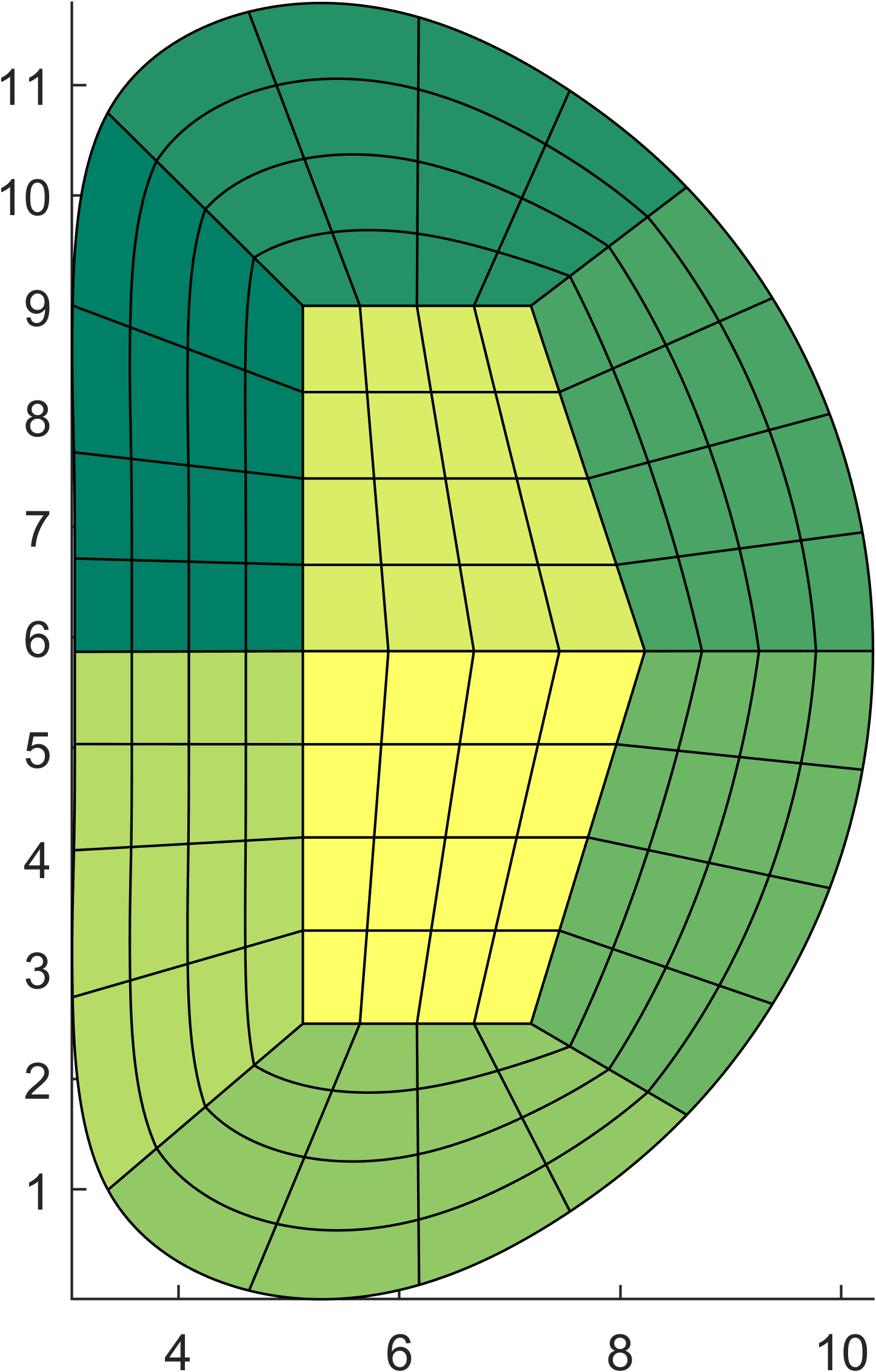}
		\caption{}
		\label{fig:iterSec}
	\end{subfigure}
	\\
	\begin{subfigure}{0.4\textwidth}
		\includegraphics[width=\textwidth]{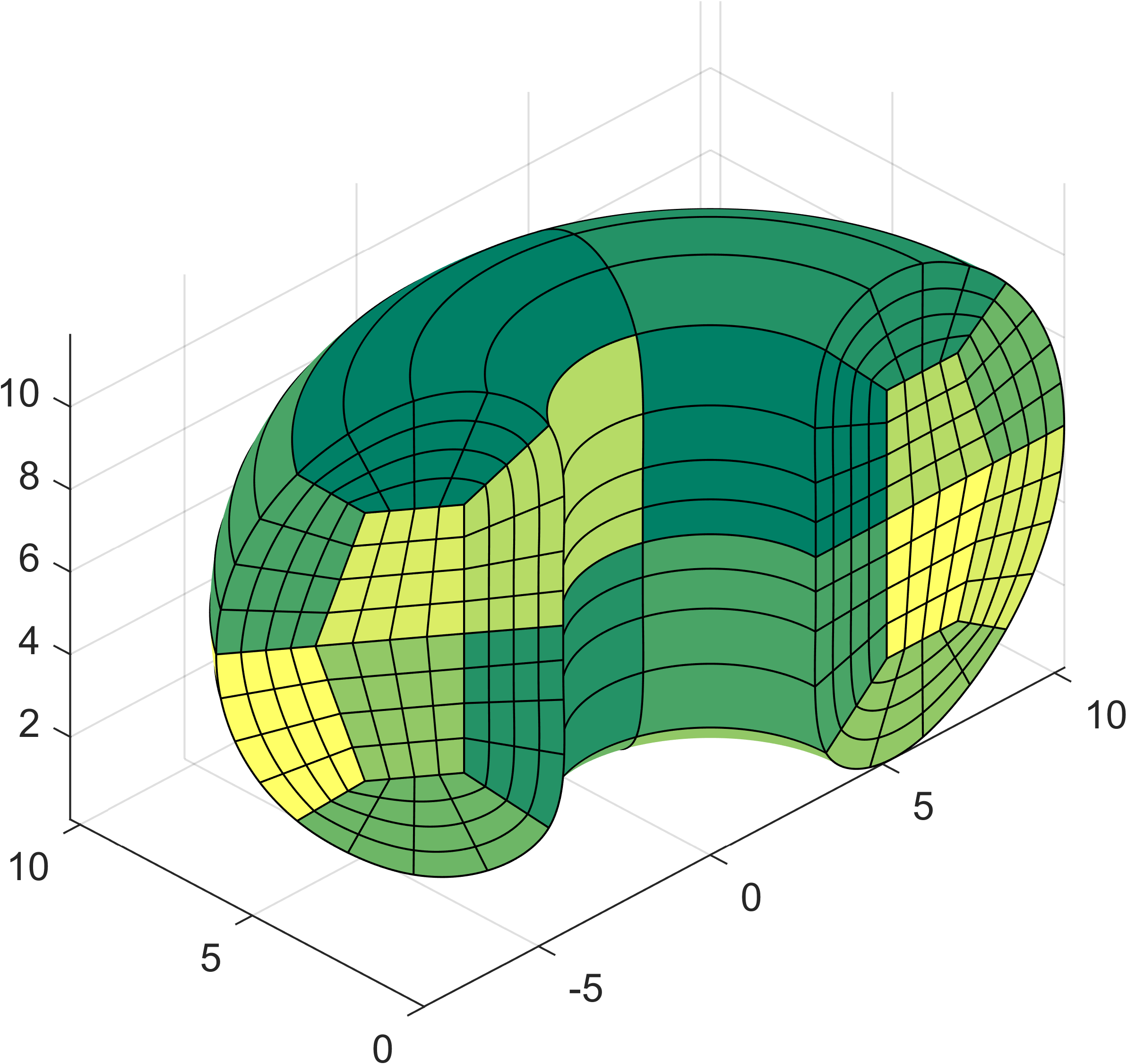}
		\caption{}
		\label{fig:iterVol2}
	\end{subfigure}
	\hfill
	\begin{subfigure}{0.4\textwidth}
		\includegraphics[width=\textwidth]{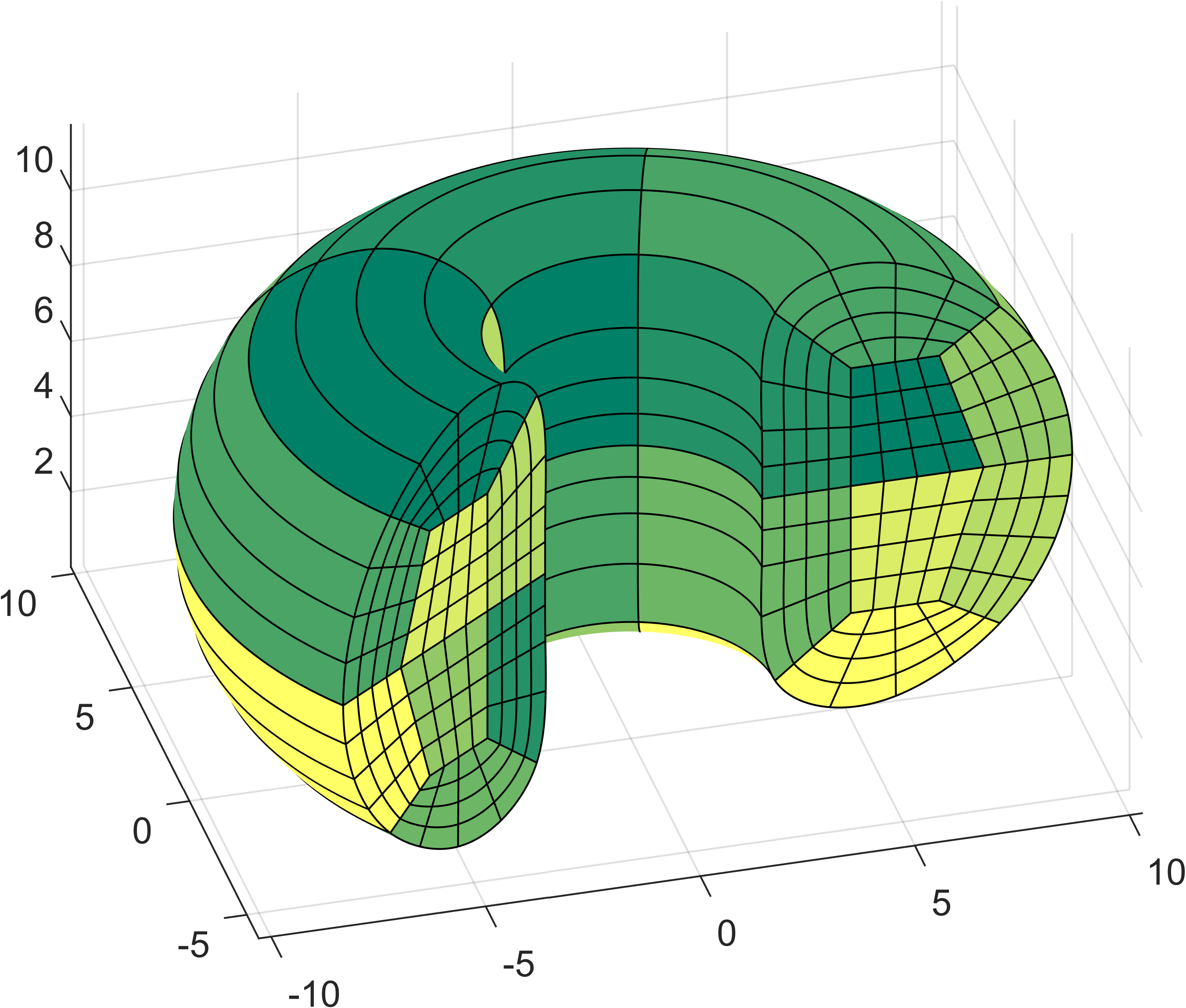}
		\caption{}
		\label{fig:iterVol3}
	\end{subfigure}
	\\
	\begin{subfigure}{0.4\textwidth}
		\includegraphics[width=\textwidth]{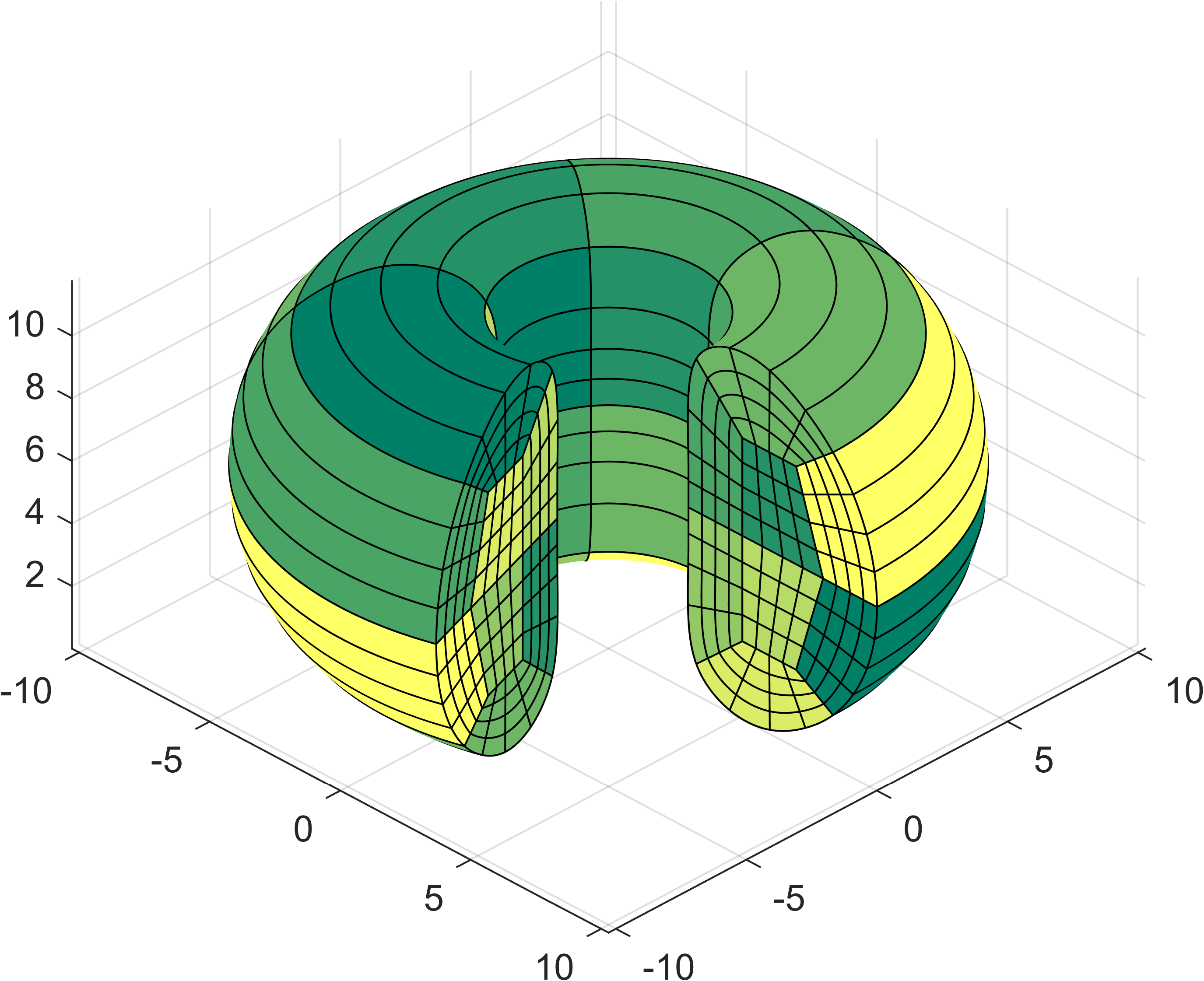}
		\caption{}
		\label{fig:iterVol4}
	\end{subfigure}
	\hfill
	\begin{subfigure}{0.4\textwidth}
		\includegraphics[width=\textwidth]{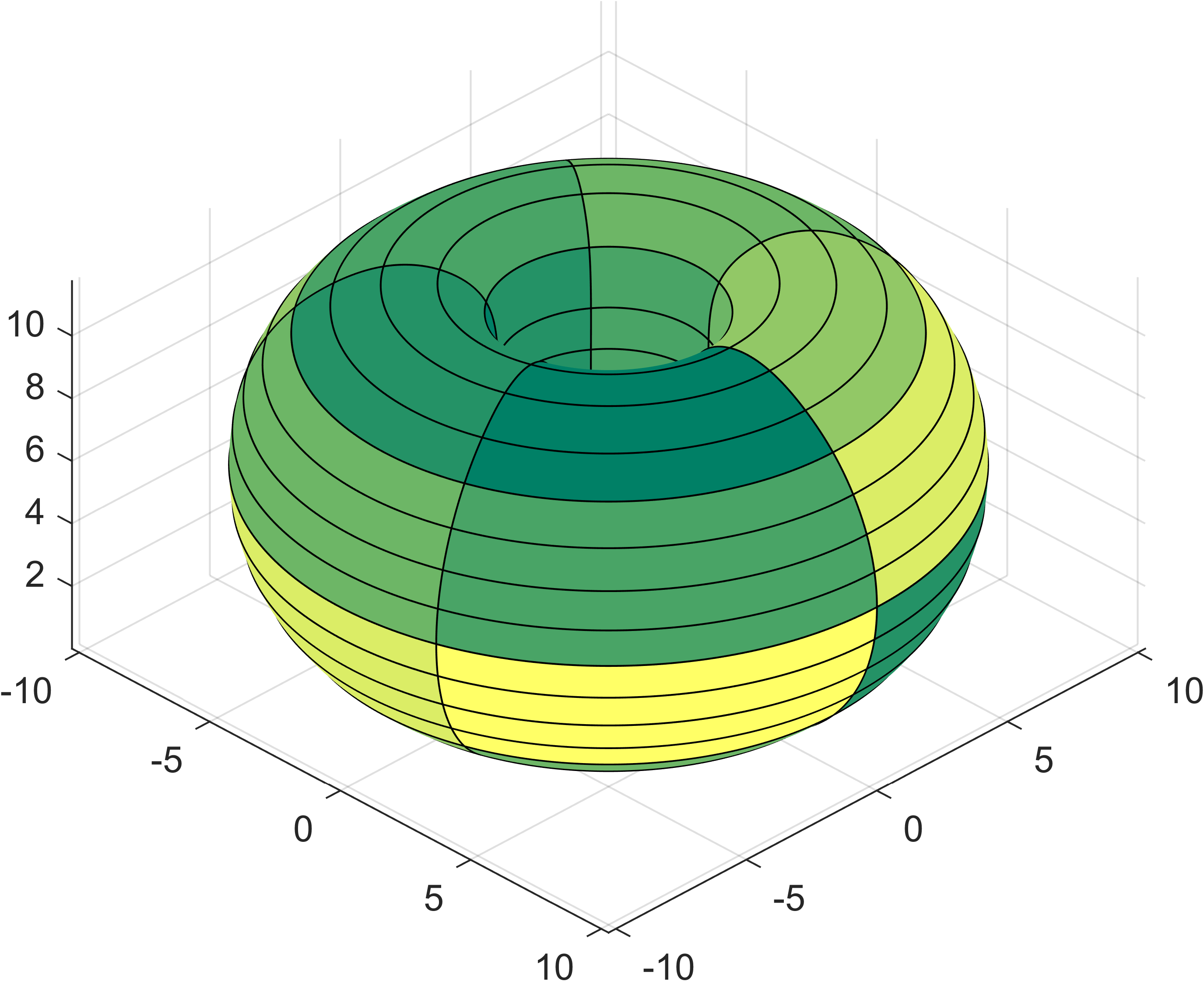}
		\caption{}
		\label{fig:iterVol5}
	\end{subfigure}
	\caption{Multi-patch domains inspired by the ITER tokamak: (a) Patch representation in section view, (b) 3D representation with 2 sectors, (c) 3D representation with 3 sectors, (d) 3D representation with 4 sectors, (e) 3D representation with 5 sectors.}
	\label{fig:ITER_mp}
\end{figure}

Table~\ref{tab:ITERvolMpE} demonstrates that the favorable performance of the proposed single-patch preconditioners extends to the additive Schwarz method within the continuous Galerkin formulation, as predicted by Theorem~\ref{thm:additive_schwarz}. Specifically, the PCG iteration count is relatively robust with respect to the number of subdivisions (it actually seems to approach a constant value as the number of subdivions grows) and exhibit only a mild dependence on the spline degree. Furthermore, the overall number of iterations remains quite low across all test cases, particularly when $\overline{\mathbf{P}}_{\overline{\mathcal{C}}}$ is employed as the local solver. We note that the largest problem size considered entails approximately $33$ million degrees of freedom.

To evaluate scalability, we compare results across the four domain configurations. Although all domains are composed of identically repeated patches, their incidence patterns differ slightly. Specifically, the maximal adjacency number is $\nadj = 12$ in the two-sector geometry, whereas it stabilizes at $\nadj = 17$ for the remaining configurations. Because we always consider at least 4 elements per parametric direction for each patch, the actual maximal interaction number is strictly bounded: it is $\nint = 14$ in the two-sector geometry and stabilizes at $\nint = 21$ in the other configurations. Ultimately, the iteration count is practically independent of the number of sectors and, consequently, of the total number of patches.

The wall-clock times for the preconditioner setup and the application of PCG, shown in Figure~\ref{fig:ITERvolMpE}, exhibit scaling behavior consistent with that observed in the single-patch test cases.
\pgfplotstableread[col sep=comma]{tables/bdPerfITERvol2slE.csv}\tabBDd
\pgfplotstableread[col sep=comma]{tables/bdPerfITERvol3slE.csv}\tabBDt
\pgfplotstableread[col sep=comma]{tables/bdPerfITERvol4slE.csv}\tabBDq
\pgfplotstableread[col sep=comma]{tables/bdPerfITERvol5slE.csv}\tabBDc
\pgfplotstableread[col sep=comma]{tables/fdPerfITERvol2slE.csv}\tabBKd
\pgfplotstableread[col sep=comma]{tables/fdPerfITERvol3slE.csv}\tabBKt
\pgfplotstableread[col sep=comma]{tables/fdPerfITERvol4slE.csv}\tabBKq
\pgfplotstableread[col sep=comma]{tables/fdPerfITERvol5slE.csv}\tabBKc

\pgfplotstablevertcat{\tabBDd}{\tabBDt}
\pgfplotstablevertcat{\tabBDd}{\tabBDq}
\pgfplotstablevertcat{\tabBDd}{\tabBDc}
\pgfplotstablevertcat{\tabBKd}{\tabBKt}
\pgfplotstablevertcat{\tabBKd}{\tabBKq}
\pgfplotstablevertcat{\tabBKd}{\tabBKc}

\def\segmentsList{}


\pgfplotstablecreatecol[create on use/segments/.style={string type}]{segments}{\tabBDd}

\begin{table}[htbp]
	\centering
	\pgfplotstabletypeset[
	columns={segments,nsub,ItsEp2,ItsEp3,ItsEp4}, 
	column type=c, 
	every head row/.style={output empty row,before row={\toprule \multirow{2}{*}{Sectors} & \multirow{2}{*}{$n_{\text{sub}}$} & \multicolumn{3}{c}{$\overline{\mathbf{P}}_{\mathbf{I}}$}\\ \cmidrule(lr){3-5} & & $p=2$ & $p=3$ & $p=4$\\}, after row=\midrule},
	every nth row={5}{before row=\midrule},
	every last row/.style={after row=\bottomrule},
	columns/segments/.style={
		assign cell content/.code={%
			\pgfmathparse{ifthenelse(int(\pgfplotstablerow/5)==\pgfplotstablerow/5,int(\pgfplotstablerow/5),-1)}%
			\ifcase\pgfmathresult
			\pgfkeyssetvalue{/pgfplots/table/@cell content}%
			{\multirow{5}{*}{2}}%
			\or
			\pgfkeyssetvalue{/pgfplots/table/@cell content}%
			{\multirow{5}{*}{3}}%
			\or
			\pgfkeyssetvalue{/pgfplots/table/@cell content}%
			{\multirow{5}{*}{4}}%
			\or
			\pgfkeyssetvalue{/pgfplots/table/@cell content}%
			{\multirow{5}{*}{5}}%
			\else
			\pgfkeyssetvalue{/pgfplots/table/@cell content}{}
			\fi
		},
	},
	columns/nsub/.style ={column name={},column type=c, int detect},
	columns/ItsEp2/.style ={column name={},column type=r},
	columns/ItsEp3/.style ={column name={},column type=r,string type},
	columns/ItsEp4/.style ={column name={},column type=r,string type}
	]\tabBDd
	\pgfplotstabletypeset[columns={ItsEp2,ItsEp3,ItsEp4}, 
	column type=c, 
	every head row/.style={output empty row,before row={\toprule \multicolumn{3}{c}{$\overline{\mathbf{P}}_{\overline{\mathcal{C}}}$}\\ \cmidrule(lr){1-3} $p=2$ & $p=3$ & $p=4$\\}, after row=\midrule},
	every nth row={5}{before row=\midrule},
	every last row/.style={after row=\bottomrule},
	columns/ItsEp2/.style ={column name={},column type=r},
	columns/ItsEp3/.style ={column name={},column type=r,string type},
	columns/ItsEp4/.style ={column name={},column type=r,string type},
	]\tabBKd
	\caption{Example 2: PCG iterations for solving a system associated with ${\mathbf{M}}_{\epsilon}$ in the multi-patch domains inspired by the ITER tokamak, as shown in Figure~\ref{fig:ITER_mp}, using the continuous Galerkin formulation. Entries marked with $*$ indicate cases where the problem exceeded the capacity of the available computational resources.
	}
	\label{tab:ITERvolMpE}
\end{table}
\tikzstyle{Linea1}=[thick,dashed]
\tikzstyle{Linea2}=[thick]
\pgfplotscreateplotcyclelist{Lista1}{%
	{Linea2,red,mark=*},
	{Linea1,red,mark=*},
	{Linea2,green,mark=triangle*},
	{Linea1,green,mark=triangle*},
	{Linea2,cyan,mark=square*},
	{Linea1,cyan,mark=square*},
	{Linea2,violet,mark=diamond*},
	{Linea1,violet,mark=diamond*},
}
\def \tabBDdu{tables/bdPerfITERvol2slE.csv}
\def \tabBKdu{tables/fdPerfITERvol2slE.csv}
\def \tabBDt{tables/bdPerfITERvol3slE.csv}
\def \tabBKt{tables/fdPerfITERvol3slE.csv}
\def \tabBDq{tables/bdPerfITERvol4slE.csv}
\def \tabBKq{tables/fdPerfITERvol4slE.csv}
\def \tabBDc{tables/bdPerfITERvol5slE.csv}
\def \tabBKc{tables/fdPerfITERvol5slE.csv}

\def\HeiFact{.78}
\begin{figure}[p]
	\centering 
	\hspace*{\fill}
	\begin{subfigure}[t]{0.4\linewidth}
		\centering
		\begin{tikzpicture}[trim axis left] 
			\begin{loglogaxis}
				[cycle list name=Lista1,
				width=\linewidth,
				height=\HeiFact\linewidth,
				xminorticks=false, yminorticks=false,
				ylabel = {\fbox{2 sectors}},
				title={Setup time [s]},
				legend columns=1,
				legend style={nodes={scale=0.85, transform shape},at={(0.01,0.99)},anchor=north west},
				xmajorgrids=true, ymajorgrids=true,
				ymin=5e-2, ymax=3e3, 
				xmin=3e3,
				xmax=9e7,
				legend style={font=\small},
				legend entries={$p=2$,,$p=3$,,$p=4$,}]
				\addplot table [x=IDofsEp2, y=TsetEp2, col sep=comma]{\tabBDdu};
				\addplot table [x=IDofsEp2, y=TsetEp2, col sep=comma]{\tabBKdu};
				\addplot table [x=IDofsEp3, y=TsetEp3, col sep=comma]{\tabBDdu};
				\addplot table [x=IDofsEp3, y=TsetEp3, col sep=comma]{\tabBKdu};
				\addplot table [x=IDofsEp4, y=TsetEp4, col sep=comma]{\tabBDdu};			
				\addplot table [x=IDofsEp4, y=TsetEp4, col sep=comma]{\tabBKdu};
				\logLogSlopeTriangle{0.85}{0.4}{0.11}{1/3}{black};
				\logLogSlopeTriangle{0.85}{0.2}{0.6}{1}{black};
			\end{loglogaxis}
		\end{tikzpicture}
	\end{subfigure}\hfill
	\begin{subfigure}[t]{0.4\linewidth}
		\centering
		\begin{tikzpicture}[
			trim axis right]
			\begin{loglogaxis}[
				cycle list name=Lista1,
				width=\linewidth, 
				height=\HeiFact\linewidth,
				xminorticks=false, yminorticks=false,
				title={PCG time [s]},
				legend columns=1,
				legend style={nodes={scale=0.85, transform shape},at={(0.01,0.99)},anchor=north west},
				xmajorgrids=true, ymajorgrids=true,
				ymin=5e-2, ymax=3e3,
				xmin=3e3,
				xmax=9e7,
				legend style={font=\small},
				legend entries={$p=2$,,$p=3$,,$p=4$,}]
				\addplot table [x=IDofsEp2, y=TprecEp2, col sep=comma]{\tabBDdu};
				\addplot table [x=IDofsEp2, y=TprecEp2, col sep=comma]{\tabBKdu};
				\addplot table [x=IDofsEp3, y=TprecEp3, col sep=comma]{\tabBDdu};
				\addplot table [x=IDofsEp3, y=TprecEp3, col sep=comma]{\tabBKdu};
				\addplot table [x=IDofsEp4, y=TprecEp4, col sep=comma]{\tabBDdu};			
				\addplot table [x=IDofsEp4, y=TprecEp4, col sep=comma]{\tabBKdu};
				\logLogSlopeTriangle{0.85}{0.2}{0.5}{1}{black};
			\end{loglogaxis}
		\end{tikzpicture}
	\end{subfigure}
	\hspace*{\fill}
	\vspace{0.3cm}
	\\
	\hspace*{\fill}
	\begin{subfigure}[t]{0.4\linewidth}
		\centering
		\begin{tikzpicture}[
			trim axis left]
			\begin{loglogaxis}[
				cycle list name=Lista1,
				width=\linewidth,
				height=\HeiFact\linewidth,
				xminorticks=false, yminorticks=false,
				ylabel = {\fbox{3 sectors}},
				xmajorgrids=true, ymajorgrids=true,
				ymin=5e-2, ymax=3e3,
				xmin=3e3,
				xmax=9e7,
				]
				\addplot table [x=IDofsEp2, y=TsetEp2, col sep=comma]{\tabBDt};
				\addplot table [x=IDofsEp2, y=TsetEp2, col sep=comma]{\tabBKt};
				\addplot table [x=IDofsEp3, y=TsetEp3, col sep=comma]{\tabBDt};
				\addplot table [x=IDofsEp3, y=TsetEp3, col sep=comma]{\tabBKt};
				\addplot table [x=IDofsEp4, y=TsetEp4, col sep=comma]{\tabBDt};			
				\addplot table [x=IDofsEp4, y=TsetEp4, col sep=comma]{\tabBKt};
				\logLogSlopeTriangle{0.85}{0.4}{0.11}{1/3}{black};
				\logLogSlopeTriangle{0.85}{0.2}{0.6}{1}{black};
			\end{loglogaxis}
		\end{tikzpicture}
	\end{subfigure}\hfill
	\begin{subfigure}[t]{0.4\linewidth}
		\centering
		\begin{tikzpicture}[
			trim axis right]
			\begin{loglogaxis}[
				cycle list name=Lista1,
				width=\linewidth,
				height=\HeiFact\linewidth,
				xminorticks=false, yminorticks=false,
				xmajorgrids=true, ymajorgrids=true,
				ymin=5e-2, ymax=3e3,
				xmin=3e3,
				xmax=9e7, 
				]
				\addplot table [x=IDofsEp2, y=TprecEp2, col sep=comma]{\tabBDt};
				\addplot table [x=IDofsEp2, y=TprecEp2, col sep=comma]{\tabBKt};
				\addplot table [x=IDofsEp3, y=TprecEp3, col sep=comma]{\tabBDt};
				\addplot table [x=IDofsEp3, y=TprecEp3, col sep=comma]{\tabBKt};
				\addplot table [x=IDofsEp4, y=TprecEp4, col sep=comma]{\tabBDt};			
				\addplot table [x=IDofsEp4, y=TprecEp4, col sep=comma]{\tabBKt};
				\logLogSlopeTriangle{0.85}{0.2}{0.5}{1}{black};
			\end{loglogaxis}
		\end{tikzpicture}
	\end{subfigure}
	\hspace*{\fill}
	\vspace{0.3cm}
	\\
	\hspace*{\fill}
	\begin{subfigure}[t]{0.4\linewidth}
		\centering
		\begin{tikzpicture}[
			trim axis left]
			\begin{loglogaxis}[
				cycle list name=Lista1,
				width=\linewidth,
				height=\HeiFact\linewidth,
				xlabel={$\ndof$},
				xminorticks=false, yminorticks=false,
				xmin=3e3,
				xmax=9e7,
				ylabel = {\fbox{4 sectors}},
				xmajorgrids=true, ymajorgrids=true,
				ymin=5e-2, ymax=3e3,
				]
				\addplot table [x=IDofsEp2, y=TsetEp2, col sep=comma]{\tabBDq};
				\addplot table [x=IDofsEp2, y=TsetEp2, col sep=comma]{\tabBKq};
				\addplot table [x=IDofsEp3, y=TsetEp3, col sep=comma]{\tabBDq};
				\addplot table [x=IDofsEp3, y=TsetEp3, col sep=comma]{\tabBKq};
				\addplot table [x=IDofsEp4, y=TsetEp4, col sep=comma]{\tabBDq};			
				\addplot table [x=IDofsEp4, y=TsetEp4, col sep=comma]{\tabBKq};
				\logLogSlopeTriangle{0.85}{0.4}{0.11}{1/3}{black};
				\logLogSlopeTriangle{0.85}{0.2}{0.6}{1}{black};
			\end{loglogaxis}
		\end{tikzpicture}
	\end{subfigure}\hfill
	\begin{subfigure}[t]{0.4\linewidth}
		\centering
		\begin{tikzpicture}[
			trim axis right]
			\begin{loglogaxis}[
				cycle list name=Lista1,
				width=\linewidth, height=\HeiFact\linewidth,
				xlabel={$\ndof$},
				xminorticks=false, yminorticks=false,
				xmajorgrids=true, ymajorgrids=true,
				xmin=3e3,
				xmax=9e7,
				ymin=5e-2, ymax=3e3,
				]
				\addplot table [x=IDofsEp2, y=TprecEp2, col sep=comma]{\tabBDq};
				\addplot table [x=IDofsEp2, y=TprecEp2, col sep=comma]{\tabBKq};
				\addplot table [x=IDofsEp3, y=TprecEp3, col sep=comma]{\tabBDq};
				\addplot table [x=IDofsEp3, y=TprecEp3, col sep=comma]{\tabBKq};
				\addplot table [x=IDofsEp4, y=TprecEp4, col sep=comma]{\tabBDq};			
				\addplot table [x=IDofsEp4, y=TprecEp4, col sep=comma]{\tabBKq};
				\logLogSlopeTriangle{0.85}{0.2}{0.5}{1}{black};
			\end{loglogaxis}
		\end{tikzpicture}
	\end{subfigure}
	\hspace*{\fill}
	\vspace{0.3cm}
	\\
	\hspace*{\fill}
	\begin{subfigure}[t]{0.4\linewidth}
		\centering
		\begin{tikzpicture}[
			trim axis left]
			\begin{loglogaxis}[
				cycle list name=Lista1,
				width=\linewidth,
				height=\HeiFact\linewidth,
				xlabel={$\ndof$},
				xminorticks=false, yminorticks=false,
				xmin=3e3,
				xmax=9e7,
				ylabel = {\fbox{5 sectors}},
				xmajorgrids=true, ymajorgrids=true,
				ymin=5e-2, ymax=3e3,
				]
				\addplot table [x=IDofsEp2, y=TsetEp2, col sep=comma]{\tabBDc};
				\addplot table [x=IDofsEp2, y=TsetEp2, col sep=comma]{\tabBKc};
				\addplot table [x=IDofsEp3, y=TsetEp3, col sep=comma]{\tabBDc};
				\addplot table [x=IDofsEp3, y=TsetEp3, col sep=comma]{\tabBKc};
				\addplot table [x=IDofsEp4, y=TsetEp4, col sep=comma]{\tabBDc};			
				\addplot table [x=IDofsEp4, y=TsetEp4, col sep=comma]{\tabBKc};
				\logLogSlopeTriangle{0.85}{0.4}{0.11}{1/3}{black};
				\logLogSlopeTriangle{0.85}{0.2}{0.6}{1}{black};
			\end{loglogaxis}
		\end{tikzpicture}
	\end{subfigure}\hfill
	\begin{subfigure}[t]{0.4\linewidth}
		\centering
		\begin{tikzpicture}[
			trim axis right]
			\begin{loglogaxis}[
				cycle list name=Lista1,
				width=\linewidth, height=\HeiFact\linewidth,
				xlabel={$\ndof$},
				xminorticks=false, yminorticks=false,
				xmajorgrids=true, ymajorgrids=true,
				xmin=3e3,
				xmax=9e7,
				ymin=5e-2, ymax=3e3,
				]
				\addplot table [x=IDofsEp2, y=TprecEp2, col sep=comma]{\tabBDc};
				\addplot table [x=IDofsEp2, y=TprecEp2, col sep=comma]{\tabBKc};
				\addplot table [x=IDofsEp3, y=TprecEp3, col sep=comma]{\tabBDc};
				\addplot table [x=IDofsEp3, y=TprecEp3, col sep=comma]{\tabBKc};
				\addplot table [x=IDofsEp4, y=TprecEp4, col sep=comma]{\tabBDc};			
				\addplot table [x=IDofsEp4, y=TprecEp4, col sep=comma]{\tabBKc};
				\logLogSlopeTriangle{0.85}{0.2}{0.5}{1}{black};
			\end{loglogaxis}
		\end{tikzpicture}
	\end{subfigure}
	\hspace*{\fill}
	\caption{Example 2: time required to solve a system associated with $\mathbf{M}_{\epsilon}$ in the multi-patch domains inspired by the ITER tokamak, as shown in Figure~\ref{fig:ITER_mp}, using the continuous Galerkin formulation. Results are presented for $\overline{\mathbf{P}}_{\mathbf{I}}$ (solid lines \rule[0.5ex]{0.45cm}{0.5pt}) and $\overline{\mathbf{P}}_{\overline{\mathcal{C}}}$ (dashed lines \rule[0.5ex]{0.15cm}{0.5pt}~\rule[0.5ex]{0.15cm}{0.5pt}~\rule[0.5ex]{0.15cm}{0.4pt}).
	}
	\label{fig:ITERvolMpE}
\end{figure}
\subsection{Example 3: Discontinuous Galerkin preconditioners for the ITER tokamak}
This example evaluates the performance of $\overline{\mathbf{P}}_{\mathbf{I}}$ and $\overline{\mathbf{P}}_{\overline{\mathcal{C}}}$ within the symmetric discontinuous Galerkin framework, applied to the multi-patch domain inspired by the ITER tokamak with five sectors (see Figure~\ref{fig:iterVol5}).

Table~\ref{tab:ITERvolMpDGE} reports the maximum number of PCG iterations across all patches. The results show that both preconditioners perform very well, demonstrating robustness with respect to both the spline degree and the number of subdivisions, in full agreement with theoretical expectations. However, $\overline{\mathbf{P}}_{\overline{\mathcal{C}}}$ consistently outperforms $\overline{\mathbf{P}}_{\mathbf{I}}$, requiring fewer PCG iterations overall.
Finally, Figure~\ref{fig:ITERvolMpDGE} reports the corresponding wall-clock times. The trends observed here are consistent with those expected and discussed in the preceding examples.
\pgfplotstableread[col sep=comma]{tables/bdPerfDGITERvol5slE.csv}\tabBD
\pgfplotstableread[col sep=comma]{tables/fdPerfDGITERvol5slE.csv}\tabBK
\begin{table}[htbp]
	\centering
	\pgfplotstabletypeset[columns={nsub,ItsEp2,ItsEp3,ItsEp4}, 
	column type=c, 
	every head row/.style={output empty row,before row={\toprule \multirow{2}{*}{$n_{\text{sub}}$} & \multicolumn{3}{c}{$\overline{\mathbf{P}}_{\mathbf{I}}$}\\ \cmidrule(lr){2-4} & $p=2$ & $p=3$ & $p=4$\\}, after row=\midrule},
	every last row/.style={after row=\bottomrule},
	columns/nsub/.style ={column name={}},
	columns/ItsEp2/.style ={column name={},column type=r},
	columns/ItsEp3/.style ={column name={},column type=r,string type},
	columns/ItsEp4/.style ={column name={},column type=r,string type}
	]\tabBD
	\pgfplotstabletypeset[columns={ItsEp2,ItsEp3,ItsEp4}, 
	column type=c, 
	every head row/.style={output empty row,before row={\toprule \multicolumn{3}{c}{$\overline{\mathbf{P}}_{\overline{\mathcal{C}}}$}\\ \cmidrule(lr){1-3} $p=2$ & $p=3$ & $p=4$\\}, after row=\midrule},
	every last row/.style={after row=\bottomrule},
	columns/ItsEp2/.style ={column name={},column type=r},
	columns/ItsEp3/.style ={column name={},column type=r,string type},
	columns/ItsEp4/.style ={column name={},column type=r,string type}
	]\tabBK
	\caption{Example 3: maximum number of PCG iterations across all patches for solving a system associated with ${\mathbf{M}}_{\epsilon}$ in the multi-patch domain inspired by the ITER tokamak with 5 sectors, as shown in Figure~\ref{fig:iterVol5}, using the symmetric discontinuous Galerkin formulation. Entries marked with $*$ indicate cases where the problem exceeded the capacity of the available computational resources.}
	\label{tab:ITERvolMpDGE}
\end{table}
\tikzstyle{Linea1}=[thick,dashed]
\tikzstyle{Linea2}=[thick]
\pgfplotscreateplotcyclelist{Lista1rs}{%
	{Linea2,red,mark=*},
	{Linea1,red,mark=*},
	{Linea2,green,mark=triangle*},
	{Linea1,green,mark=triangle*},
	{Linea2,cyan,mark=square*},
	{Linea1,cyan,mark=square*},
	{Linea2,violet,mark=diamond*},
	{Linea1,violet,mark=diamond*},
}
\def \tabBD{tables/bdPerfDGITERvol5sl.csv}
\def \tabBK{tables/fdPerfDGITERvol5sl.csv}
\begin{figure}[htbp]
	\centering
	\hspace*{\fill}
	\begin{subfigure}[t]{0.4\linewidth}
		\centering
		\begin{tikzpicture}[
			trim axis left
			]
			\begin{loglogaxis}[
				cycle list name=Lista1rs,
				width=\linewidth,
				height=\linewidth,
				xlabel={$\ndof$},
				ymin=1e-1,
				ymax=2e3,
				xminorticks=false,
				yminorticks=false,
				ylabel={Setup time [s]},
				legend columns=1,
				legend style={at={(0.01,0.99)},anchor=north west},
				xmajorgrids=true,
				ymajorgrids=true,
				legend entries={$p=2$,,$p=3$,,$p=4$,}
				]
				\addplot table [x=IDofsEp2, y=TsetEp2, col sep=comma]{\tabBD};
				\addplot table [x=IDofsEp2, y=TsetEp2, col sep=comma]{\tabBK};
				\addplot table [x=IDofsEp3, y=TsetEp3, col sep=comma]{\tabBD};
				\addplot table [x=IDofsEp3, y=TsetEp3, col sep=comma]{\tabBK};
				\addplot table [x=IDofsEp4, y=TsetEp4, col sep=comma]{\tabBD};			
				\addplot table [x=IDofsEp4, y=TsetEp4, col sep=comma]{\tabBK};
				\logLogSlopeTriangle{0.85}{0.5}{0.1}{1/3}{black};
				\logLogSlopeTriangle{0.85}{0.2}{0.6}{1}{black};
			\end{loglogaxis}
		\end{tikzpicture}
	\end{subfigure}
	\hfill
	\begin{subfigure}[t]{0.4\linewidth}
		\centering
		\begin{tikzpicture}[
			trim axis right
			]
			\begin{loglogaxis}[
				cycle list name=Lista1rs,
				width=\linewidth,
				height=\linewidth,
				xlabel={$\ndof$},
				ymin=1e-1,
				ymax=2e3,
				xminorticks=false,
				yminorticks=false,
				ylabel={PCG time [s]},
				legend columns=1,
				legend style={at={(0.01,0.99)},anchor=north west},
				xmajorgrids=true,
				ymajorgrids=true,
				legend entries={$p=2$,,$p=3$,,$p=4$,}
				]
				\addplot table [x=IDofsEp2, y=TprecEp2, col sep=comma] {\tabBD};
				\addplot table [x=IDofsEp2, y=TprecEp2, col sep=comma] {\tabBK};
				\addplot table [x=IDofsEp3, y=TprecEp3, col sep=comma] {\tabBD};
				\addplot table [x=IDofsEp3, y=TprecEp3, col sep=comma] {\tabBK};
				\addplot table [x=IDofsEp4, y=TprecEp4, col sep=comma] {\tabBD};				
				\addplot table [x=IDofsEp4, y=TprecEp4, col sep=comma] {\tabBK};
				\logLogSlopeTriangle{0.85}{0.2}{0.45}{1}{black};
			\end{loglogaxis}
		\end{tikzpicture}
	\end{subfigure}
	\hspace*{\fill}
	\caption{Example 3: time required to solve a system associated with $\mathbf{M}_{\epsilon}^1$ in the multi-patch domain inspired by the ITER tokamak with 5 sectors, as shown in Figure~\ref{fig:iterVol5}, using the symmetric discontinuous Galerkin formulation. Results are presented for $\overline{\mathbf{P}}_{\mathbf{I}}$ (solid lines \rule[0.5ex]{0.45cm}{0.5pt}) and $\overline{\mathbf{P}}_{\overline{\mathcal{C}}}$ (dashed lines \rule[0.5ex]{0.15cm}{0.5pt}~\rule[0.5ex]{0.15cm}{0.5pt}~\rule[0.5ex]{0.15cm}{0.4pt}).}
	\label{fig:ITERvolMpDGE}
\end{figure}
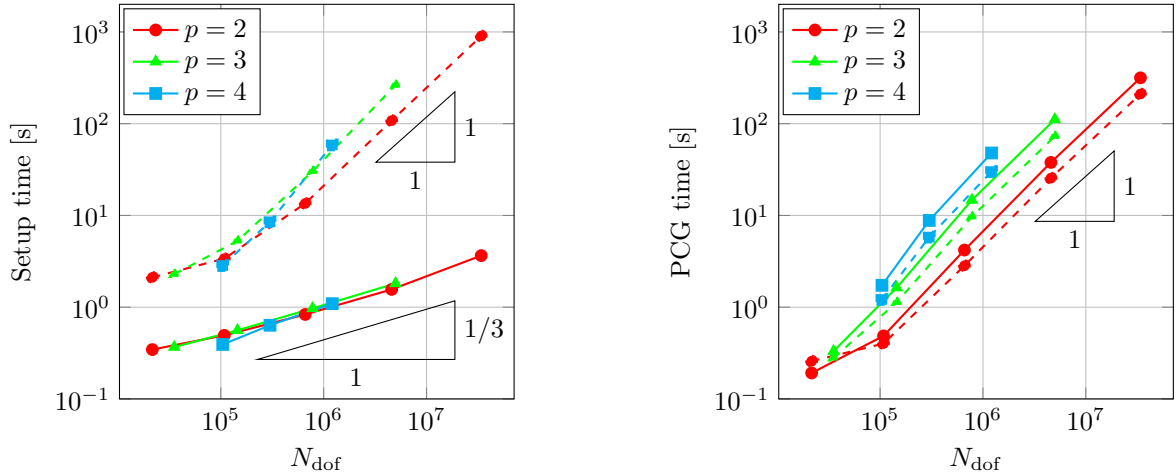
\subsection{Example 4: Discontinuous Galerkin preconditioners for the twisted cable}
In this second test within the symmetric discontinuous Galerkin framework, we evaluate the preconditioners on a multi-patch twisted cable domain, illustrated in Figure~\ref{fig:cable}. The three-dimensional geometry is obtained by extruding the two-dimensional cross-section, shown in Figure~\ref{fig:cableSec}, along a helical path. The helix completes a total rotation of $\pi/4$ radians and extends to a height of \num{1}. Notably, the extrusion path is not perpendicular to the sweeping cross-section. Consequently, the geometric mappings do not satisfy the orthogonality condition for the derivative in the third parametric direction with respect to the other two, which would be favorable for the block Kronecker preconditioner, see Section~\ref{sec:fd_prec}
\begin{figure}[htbp]
	\centering
	\begin{subfigure}{0.49\textwidth}
		\centering
		\includegraphics[height=0.5\textwidth]{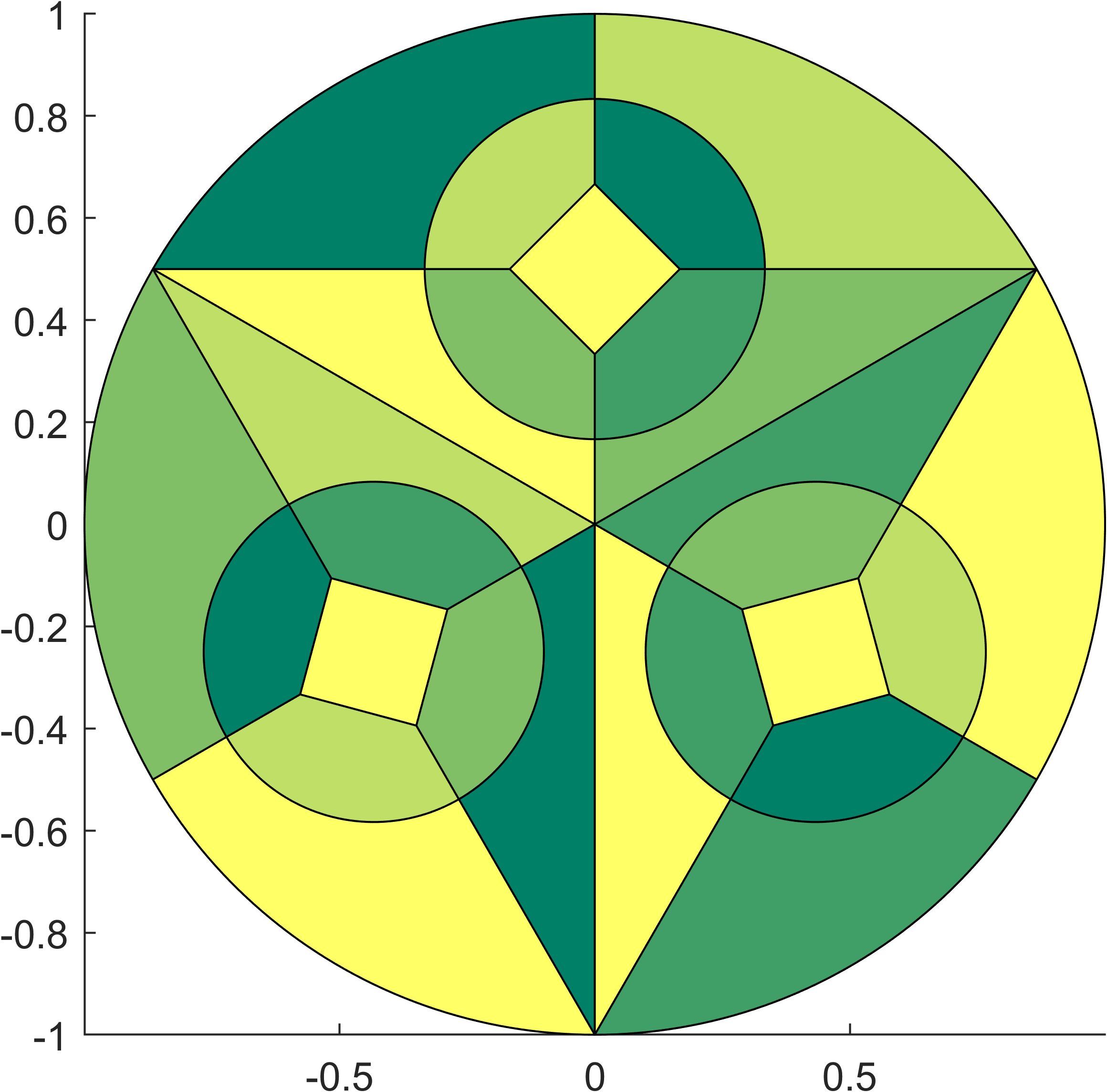}
		\caption{}
		\label{fig:cableSec}
	\end{subfigure}
	\hfill
	\begin{subfigure}{0.49\textwidth}
		\centering
		\includegraphics[height=0.5\textwidth]{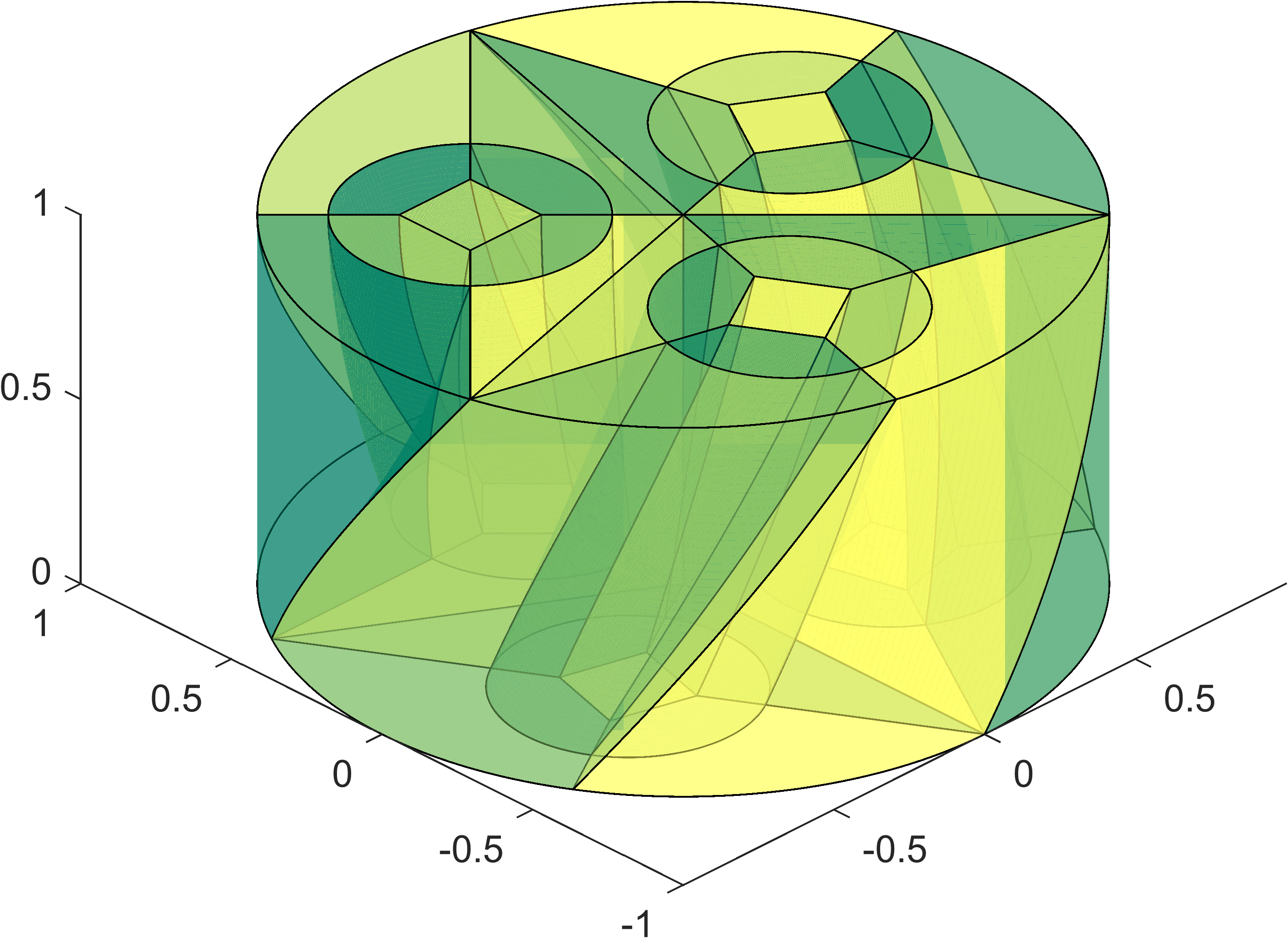}
		\caption{}
	\end{subfigure}
	\caption{Multi-patch twisted cable domain: (a) patch representation in section view, (b) 3D representation.}
	\label{fig:cable}
\end{figure}

Table~\ref{tab:cableE} reports the maximum number of PCG iterations across all patches for this domain. Both preconditioners demonstrate robust performance and yield very similar iteration counts. Since $\overline{\mathbf{P}}_{\overline{\mathcal{C}}}$ is specifically designed for orthogonal geometries, its iteration count is slightly higher in this example than in previous cases where the orthogonality assumption holds. Nevertheless, it still retains a slight advantage over $\overline{\mathbf{P}}_{\mathbf{I}}$. 
We note that the largest problem size considered entails approximately $22$ million degrees of freedom.
Finally, Figure~\ref{fig:cableE} presents the wall-clock times, which exhibit trends consistent with those observed in the previous test cases.
\pgfplotstableread[col sep=comma]{tables/bdPerfCableE.csv}\tabBD
\pgfplotstableread[col sep=comma]{tables/fdPerfCableE.csv}\tabBK
\begin{table}[htbp]
\centering
\pgfplotstabletypeset[columns={nsub,ItsEp2,ItsEp3,ItsEp4}, 
column type=c, 
every head row/.style={output empty row,before row={\toprule \multirow{2}{*}{$n_{\text{sub}}$} & \multicolumn{3}{c}{$\overline{\mathbf{P}}_{\mathbf{I}}$}\\ \cmidrule(lr){2-4} & $p=2$ & $p=3$ & $p=4$\\}, after row=\midrule},
every last row/.style={after row=\bottomrule},
columns/nsub/.style ={column name={}},
columns/ItsEp2/.style ={column name={},column type=r},
columns/ItsEp3/.style ={column name={},column type=r,string type},
columns/ItsEp4/.style ={column name={},column type=r,string type}
]\tabBD
\pgfplotstabletypeset[columns={ItsEp2,ItsEp3,ItsEp4}, 
column type=c, 
every head row/.style={output empty row,before row={\toprule \multicolumn{3}{c}{$\overline{\mathbf{P}}_{\overline{\mathcal{C}}}$}\\ \cmidrule(lr){1-3} $p=2$ & $p=3$ & $p=4$\\}, after row=\midrule},
every last row/.style={after row=\bottomrule},
columns/ItsEp2/.style ={column name={},column type=r},
columns/ItsEp3/.style ={column name={},column type=r,string type},
columns/ItsEp4/.style ={column name={},column type=r,string type}
]\tabBK
\caption{Example 4: maximum number of PCG iterations across all patches for solving a system associated with ${\mathbf{M}}_{\epsilon}$ in the multi-patch twisted cable domain, as shown in Figure~\ref{fig:cable}, using the symmetric discontinuous Galerkin formulation. Entries marked with $*$ indicate cases where the problem exceeded the capacity of the available computational resources.}
\label{tab:cableE}
\end{table}
\tikzstyle{Linea1}=[thick,dashed]
\tikzstyle{Linea2}=[thick]
\pgfplotscreateplotcyclelist{Lista1rs}{%
	{Linea2,red,mark=*},
	{Linea1,red,mark=*},
	{Linea2,green,mark=triangle*},
	{Linea1,green,mark=triangle*},
	{Linea2,cyan,mark=square*},
	{Linea1,cyan,mark=square*},
	{Linea2,violet,mark=diamond*},
	{Linea1,violet,mark=diamond*},
}
\def \tabBD{tables/bdPerfCableE.csv}
\def \tabBK{tables/fdPerfCableE.csv}
\begin{figure}[htbp]
	\centering
	\hspace*{\fill}
	\begin{subfigure}[t]{0.4\linewidth}
		\centering
		\begin{tikzpicture}[
			trim axis left
			]
			\begin{loglogaxis}[
				cycle list name=Lista1rs,
				width=\linewidth,
				height=\linewidth,
				xlabel={$\ndof$},
				ymin=1e-1,
				ymax=1e3,
				xminorticks=false,
				yminorticks=false,
				ylabel={Setup time [s]},
				legend columns=1,
				legend style={at={(0.01,0.99)},anchor=north west},
				xmajorgrids=true,
				ymajorgrids=true,
				legend entries={$p=2$,,$p=3$,,$p=4$,}
				]
				\addplot table [x=IDofsEp2, y=TsetEp2, col sep=comma]{\tabBD};
				\addplot table [x=IDofsEp2, y=TsetEp2, col sep=comma]{\tabBK};
				\addplot table [x=IDofsEp3, y=TsetEp3, col sep=comma]{\tabBD};
				\addplot table [x=IDofsEp3, y=TsetEp3, col sep=comma]{\tabBK};
				\addplot table [x=IDofsEp4, y=TsetEp4, col sep=comma]{\tabBD};			
				\addplot table [x=IDofsEp4, y=TsetEp4, col sep=comma]{\tabBK};
				\logLogSlopeTriangle{0.85}{0.5}{0.1}{1/3}{black};
				\logLogSlopeTriangle{0.85}{0.2}{0.55}{1}{black};
			\end{loglogaxis}
		\end{tikzpicture}
	\end{subfigure}
	\hfill
	\begin{subfigure}[t]{0.4\linewidth}
		\centering
		\begin{tikzpicture}[
			trim axis right
			]
			\begin{loglogaxis}[
				cycle list name=Lista1rs,
				width=\linewidth,
				height=\linewidth,
				xlabel={$\ndof$},
				ymin=1e-1,
				ymax=1e3,
				xminorticks=false,
				yminorticks=false,
				ylabel={PCG time [s]},
				legend columns=1,
				legend style={at={(0.01,0.99)},anchor=north west},
				xmajorgrids=true,
				ymajorgrids=true,
				legend entries={$p=2$,,$p=3$,,$p=4$,}
				]
				\addplot table [x=IDofsEp2, y=TprecEp2, col sep=comma] {\tabBD};
				\addplot table [x=IDofsEp2, y=TprecEp2, col sep=comma] {\tabBK};
				\addplot table [x=IDofsEp3, y=TprecEp3, col sep=comma] {\tabBD};
				\addplot table [x=IDofsEp3, y=TprecEp3, col sep=comma] {\tabBK};
				\addplot table [x=IDofsEp4, y=TprecEp4, col sep=comma] {\tabBD};				
				\addplot table [x=IDofsEp4, y=TprecEp4, col sep=comma] {\tabBK};
				\logLogSlopeTriangle{0.85}{0.2}{0.45}{1}{black};
			\end{loglogaxis}
		\end{tikzpicture}
	\end{subfigure}
	\hspace*{\fill}
	\caption{Example 4: time required to solve a system associated with $\mathbf{M}_{\epsilon}^1$ in the multi-patch twisted cable domain as shown in Figure~\ref{fig:cable}, using the symmetric discontinuous Galerkin formulation. Results are presented for $\mathbf{P}_{\mathbf{I}}$ (solid lines \rule[0.5ex]{0.45cm}{0.5pt}) and $\mathbf{P}_{\overline{\mathcal{C}}}$ (dashed lines \rule[0.5ex]{0.15cm}{0.5pt}~\rule[0.5ex]{0.15cm}{0.5pt}~\rule[0.5ex]{0.15cm}{0.4pt}).}
	\label{fig:cableE}
\end{figure}
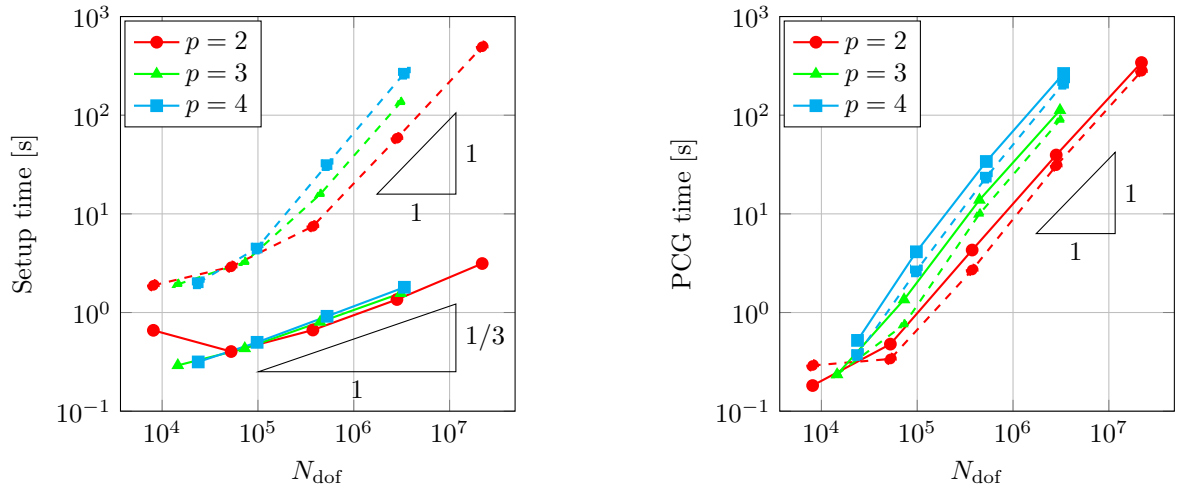
\section{Conclusions} \label{sec:conclusions}
In this paper, we introduced a preconditioning strategy for the mass matrix in spaces of 1-forms, which arises in the explicit time integration of Maxwell's equations. We first analyzed a class of spectrally equivalent preconditioners in the context of a single-patch domain, proposing two distinct approaches that are both robust and computationally efficient. We then extended these preconditioners to the multi-patch setting by incorporating them as local solvers within an additive Schwarz framework. This extension ensures that the approach remains efficient, robust, and scalable with respect to the total number of patches.

Moreover, the inherent structure of the single-patch preconditioners, combined with the additive Schwarz framework, facilitates a highly parallel implementation, significantly improving the efficiency of the preconditioning techniques. To validate our approach, we conducted numerical experiments on realistic problems, which confirmed the theoretical findings and demonstrated the practical effectiveness of the proposed methods.

\section*{Acknowledgements}
G. Loli, G. Sangalli, and M. Tani are members of the Gruppo Nazionale Calcolo Scientifico-Istituto Nazionale di Alta Matematica (GNCS-INDAM).
G. Sangalli and M. Tani acknowledge the support of the Italian Ministry of University and Research (MUR) through the PRIN 2022 PNRR project NOTES (No. P2022NC97R), funded by the European Union - NextGenerationEU.
This research also received financial support from ICSC - the Italian Research Center on High-Performance Computing, Big Data, and Quantum Computing, funded by the European Union - NextGenerationEU. The work of Rafael Vázquez has been partially funded by Consellería de Educación, Ciencia, Universidades e Formación Profesional - Xunta de Galicia (2025-AD059 and ED431F 2025/03), and by the Spanish State Research Agency (PID2024-156071NB-I00).
\begin{figure}[H]
	{\centering
		\hfill\includegraphics[scale=0.075]{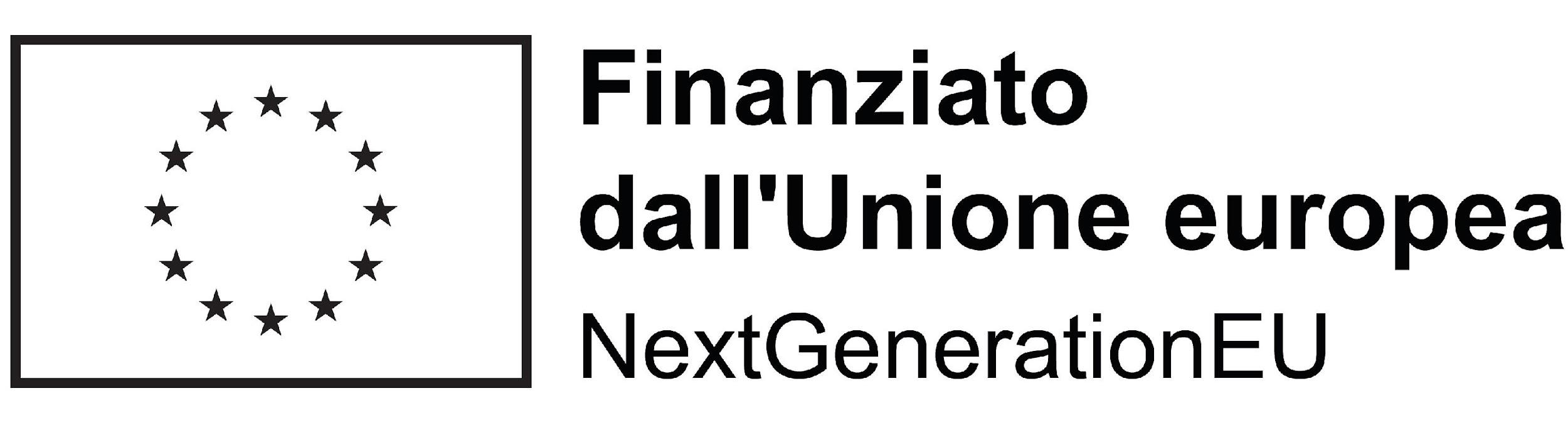}
		\hfill\includegraphics[scale=0.075]{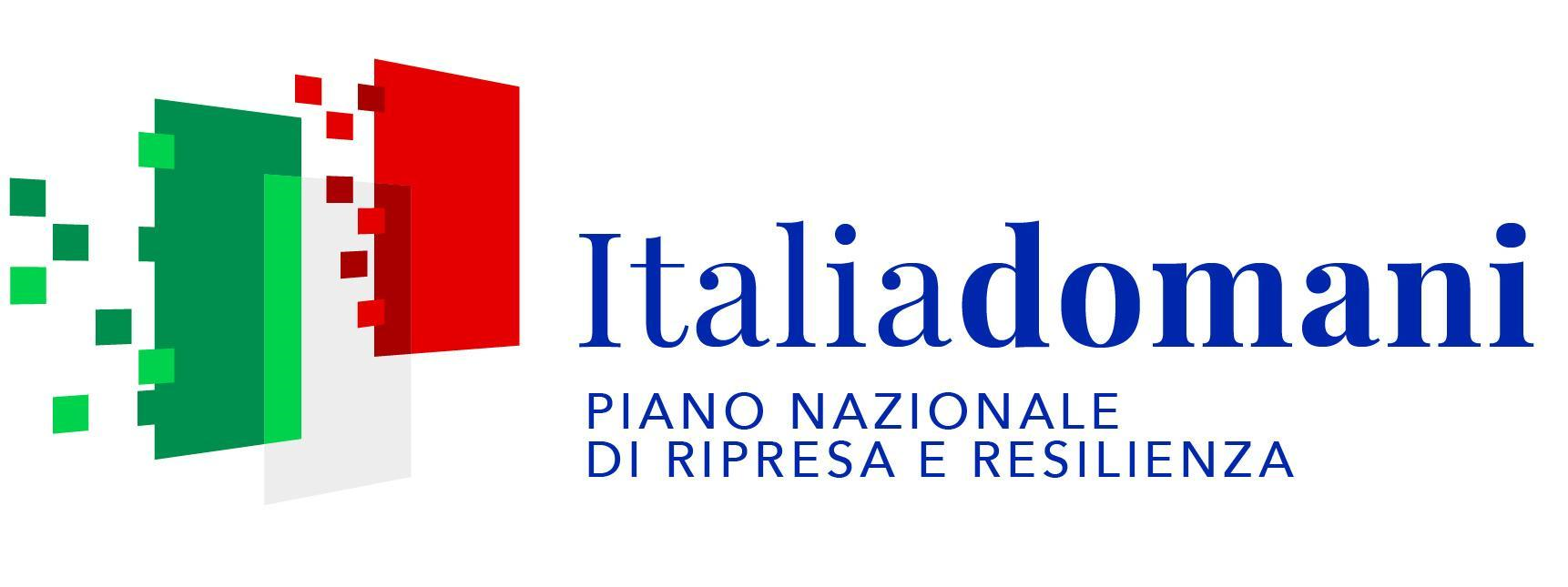}
		\hfill\includegraphics[scale=0.075]{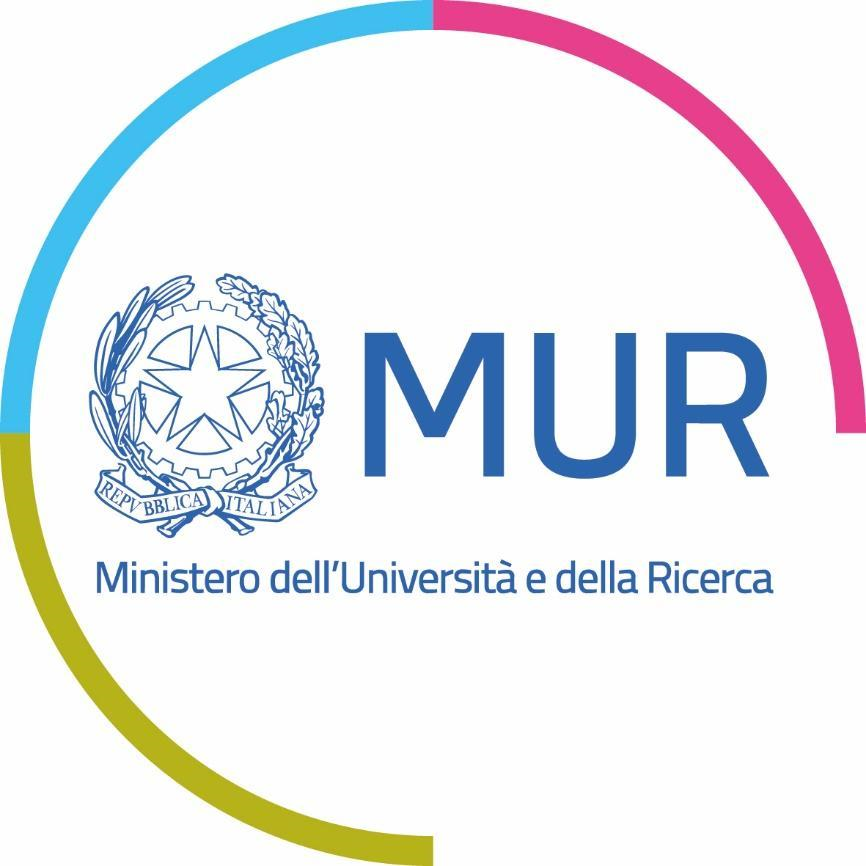}\hfill\mbox{}}
\end{figure}
\printbibliography
 \end{document}